\documentclass[11pt,reqno,a4paper]{amsart}
\usepackage[T1]{fontenc}
\usepackage[utf8]{inputenc}
\usepackage{fbb}

\usepackage{amsmath,amsthm,amssymb,mathtools}
\allowdisplaybreaks
\usepackage[top=2.5cm,bottom=2.4cm,left=2.6cm,right=2.6cm,headsep=0.2in]{geometry}
\usepackage{fancyhdr}
\usepackage{microtype}
\usepackage{enumitem}
\usepackage[noadjust]{cite}
\usepackage{xcolor}
\usepackage{hyperref}
\hypersetup{
  colorlinks=true,
  linkcolor=blue!50!green,
  citecolor=blue!50!green,
  urlcolor=blue!50!green,
  pdftitle={Mass non-concentration at the nodal set and sharp Wasserstein bounds for eigenfunctions},
  pdfauthor={Mayukh Mukherjee}
}
\newtheorem{theorem}{Theorem}[section]
\newtheorem{proposition}[theorem]{Proposition}
\newtheorem{lemma}[theorem]{Lemma}
\newtheorem{corollary}[theorem]{Corollary}
\theoremstyle{remark}
\newtheorem{remark}[theorem]{Remark}
\newcommand{\RR}{\mathbb R}

\newcommand{\HH}{\mathcal H}
\newcommand{\Vol}{\operatorname{Vol}}
\newcommand{\dist}{\operatorname{dist}}
\newcommand{\diam}{\operatorname{diam}}
\newcommand{\supp}{\operatorname{supp}}
\newcommand{\T}{\mathcal T}

\newcommand{\Zreg}{Z_{\mathrm{reg}}}
\newcommand{\dV}{\,dV_g}

\newcommand{\norm}[1]{\left\lVert#1\right\rVert}

\newcommand{\ip}[2]{\left\langle#1,#2\right\rangle}
\title[Mass non-concentration and Wasserstein distance]{Mass non-concentration at the nodal set and sharp Wasserstein bounds for eigenfunctions}
\author{Mayukh Mukherjee}
\address{Indian Institute of Technology Bombay, Powai, Maharashtra 400076, India}
\email{mukherjee@math.iitb.ac.in, mathmukherjee@gmail.com}
\date{November 25, 2026.}
\begin{document}
\begin{abstract}
Let $\varphi$ be a real Laplace eigenfunction of eigenvalue $\lambda>0$ on a smooth closed connected Riemannian manifold $(M,g)$ of dimension $n\ge2$. Put $h=\lambda^{-1/2}$ and $\varphi^\pm=\max(\pm\varphi,0)$. There is a probability measure on $M\times M$ with marginals proportional to $\varphi^+\,dV_g$ and $\varphi^-\,dV_g$, supported on pairs at distance at most $C_g h$. In particular, their Wasserstein distance has order $h$ for every $1\le p\le\infty$. Each marginal can also be coupled to the normalized measure $|\nabla\varphi|\,dS_g$ on the nodal set within distance $ah-\frac{a}{12}(\min_M\mathrm{Scal}_g)h^3+O_g(h^4)$ for small $h$, where $a$ is the first positive zero of $J_{(n-2)/2}$. The first two terms are attained by zonal harmonics on the round sphere.

Over all couplings of the positive and negative parts, the largest fraction of mass that can travel at most $sh$ is $s^2/4+O_g(s^3)$ for small $s$. For $0<p\le\infty$ we bound the $L^p$ norm of each part outside a sufficiently thin nodal tube from below by a fixed multiple of $\|\varphi\|_p$. For finite $p$ we also find the leading term of the $L^p$ mass on each side inside the tube. For small $h$, each part accounts for $\frac12+O_g(h^2)$ of the $L^1$ mass in every ball of radius $j_{n/2,1}h$, where $j_{n/2,1}$ is the first positive zero of $J_{n/2}$. For sums of eigenfunctions a single coupling gives bounds on the mass that travels far, and for finite sums it gives the bound on zero-free balls asked for by Decio.
\end{abstract}
\maketitle
\section{Introduction}

Let $(M,g)$ be a smooth closed connected Riemannian manifold of dimension $n\geq2$. We take $-\Delta_g$ to be nonnegative and consider a nonzero real eigenfunction $\varphi$,
\begin{equation}\label{eq:eigen}
 -\Delta_g\varphi=\lambda\varphi,\qquad \lambda>0.
\end{equation}
We write $h=\lambda^{-1/2}$ for the wavelength, $Z=\{\varphi=0\}$ for the nodal set, and $T_r=\{x:d_g(x,Z)<r\}$ for its open $r$-neighbourhood. Let $\varphi^+=\max(\varphi,0)$ and $\varphi^-=\max(-\varphi,0)$. The two measures
\[
 \mu_\pm=\varphi^\pm\,dV_g
\]
have the same mass $m=\|\varphi\|_1/2$, since $\varphi$ has mean zero. We write $\nu_\pm=\mu_\pm/m$ for the corresponding probability measures. Constants with subscript $g$ depend only on $(M,g)$, and so does the threshold when we take $h$ sufficiently small.

\subsection{Wasserstein bounds}
A coupling of two finite nonnegative measures $\alpha,\beta$ of equal mass is a measure on $M\times M$ with first marginal $\alpha$ and second marginal $\beta$. We write $\Pi(\alpha,\beta)$ for the set of these couplings. For $1\leq p<\infty$ put
\[
 W_p(\alpha,\beta)=\inf_{\pi\in\Pi(\alpha,\beta)}
 \left(\int d_g(x,y)^p\,d\pi(x,y)\right)^{1/p}.
\]
For $p=\infty$ we minimize the essential supremum of the distance travelled.

Steinerberger conjectured that $W_p(\nu_+,\nu_-)\asymp_{g,p} h$ for every $1\leq p<\infty$. His upper bound has an additional factor of order $\sqrt{\log\lambda}$ at high frequency \cite[Section~3.3, p.~317, and Theorem~5, p.~318]{StFourier}. Carroll, Massaneda and Ortega-Cerd\`a \cite[Theorem~3, p.~1160]{CMO} proved the sharp upper bound for $W_1$, including for sums of high frequencies.

In an earlier version of this paper the author proved the sharp $W_1$ lower bound on surfaces and gave a weaker bound in higher dimensions. De Ponti and Farinelli subsequently obtained the sharp lower bound in all dimensions by a heat-flow argument \cite[Theorem~1.2, p.~4]{DPF}. A later version of this paper obtained the same bound by a nodal-set argument \cite[Theorem~2.4, p.~3]{Muk}. The heat-flow estimate also gives the lower bound for $W_\infty$ \cite[Remark~4.3, p.~15]{DPF}. On a closed manifold one may take the spectral cutoff in their theorem to be the first positive eigenvalue. On a finite-measure $\mathsf{RCD}(K,\infty)$ space $(X,d,\mathfrak m)$ their estimates apply for $\lambda\geq\lambda_0>0$ provided $\int_X d(x,x_0)|\varphi(x)|\,d\mathfrak m(x)<\infty$ for some $x_0\in X$ \cite[Theorem~1.4, pp.~4--5, and Corollary~4.2, p.~15]{DPF}.

We obtain the upper bounds for $1<p\leq\infty$ by constructing a single coupling supported on pairs at distance at most $C_g h$.

\begin{samepage}
\begin{theorem}[Transport at the wavelength scale]\label{new:endpoint}
There is a coupling $\pi\in\Pi(\nu_+,\nu_-)$ supported on
$\{(x,y):d_g(x,y)\leq C_g h\}$. Together with the lower bound proved in Section~\ref{sec:thm_uncer}, this gives
\[
 c_g h\leq W_1(\nu_+,\nu_-)\leq W_p(\nu_+,\nu_-)
 \leq W_\infty(\nu_+,\nu_-)\leq C_g h,\qquad 1\leq p\leq\infty .
\]
Let $j_{\alpha,1}$ denote the first positive zero of the Bessel function $J_\alpha$ of the first kind. For sufficiently small $h$,
\begin{equation}\label{eq:improved-radius}
 W_\infty(\nu_+,\nu_-)\leq 2j_{(n-2)/2,1}h+C_gh^3.
\end{equation}
\end{theorem}
\end{samepage}
On the unit round $S^3$ the leading term in \eqref{eq:improved-radius} is $2\pi h$. For a zonal harmonic of degree $k\ge1$ we find $W_\infty=\pi/(k+1)$, which is asymptotic to $\pi h$ as $k\to\infty$ (see \eqref{eq:zonal-three-sphere}).

For $\varphi(x,y)=\sin(kx)$ on the flat two-torus, translation by $\pi/k$ carries $\nu_+$ to $\nu_-$. Points near the middle of a strip are at distance comparable to $1/k$ from the opposite sign, while points on either side of a nodal line can be much closer. The next theorem gives the largest mass a coupling can assign to pairs at small distances.

\begin{theorem}[Short-distance transport]\label{thm:transport}
For sufficiently small $s>0$,
\begin{equation}\label{eq:transport-main}
 \sup_{\pi\in\Pi(\nu_+,\nu_-)}\pi\{(x,y):d_g(x,y)\leq sh\}
 =\frac{s^2}{4}+O_g(s^3).
\end{equation}
The error bound and the range of $s$ depend only on $(M,g)$.
\end{theorem}

The coupling in \eqref{eq:transport-main} may depend on $s$ and need not minimize $W_p$. The power $s^2$ comes from linear vanishing in the normal direction; the construction in Section~\ref{sec:transport} gives the coefficient $1/4$.

\subsection{Ball averages}
Suppose that averaging $\mu_+$ and $\mu_-$ over balls of radius $r$ gives the same measure. At each centre we take the two conditional distributions of the original points given that centre. Integrating their product against the common measure gives a coupling. Both points are within $r$ of the centre, so their distance apart is at most $2r$.

Put $\nu=(n-2)/2$ and let $F_\alpha(s)=2^\alpha\Gamma(\alpha+1)s^{-\alpha}J_\alpha(s)$ for $s>0$, with $F_\alpha(0)=1$. Normalized volume on the Euclidean unit ball has Fourier transform $F=F_{n/2}$, and normalized area on its boundary has transform $F_\nu$. For small $r$, $K_r=F(r\sqrt{-\Delta_g})$ is a positive average over the ball of radius $r$ about each centre (Lemma~\ref{lem:ball-kernel}). Since $K_r\varphi=F(r/h)\varphi$, we get equal averages at $r=jh$, where $j=j_{n/2,1}$ is the first positive zero of $F$.

On the regular part $\Zreg=\{z\in Z:\nabla\varphi(z)\ne0\}$, let
\[
 d\sigma=h^2|\nabla\varphi|\,dS_g\big|_{\Zreg}.
\]
Here $dS_g$ is $(n-1)$-dimensional Hausdorff measure and $|\nabla\varphi|$ is the normal derivative towards the positive side. The singular zeros have zero $dS_g$-measure by Hardt and Simon \cite[Theorem~(1.10)]{HS}. B\"ar also describes the singular set as having smaller dimension \cite[Abstract]{B}. The nodal identity of Sogge and Zelditch \cite[Proposition~1, Eq.~(3), p.~26]{SZ} shows that $\sigma$ has the same mass $m$ as $\mu_+$ and $\mu_-$.

For an averaging operator $A$, write $A^*$ for its action on measures by duality. The largest measure dominated by both $K_r^*\mu_+$ and $K_r^*\mu_-$ has density
\[
 C_r(x)=\min\{K_r\varphi^+(x),K_r\varphi^-(x)\}.
\]

\begin{theorem}[Local cancellation]\label{thm:cancellation}
There is a radius $r_g>0$, depending only on $(M,g)$, such that for $0<r<r_g$,
\begin{equation}\label{eq:support-and-mass}
 \{C_r>0\}=T_r,\qquad
 \int_M C_r\,dV_g=m\bigl(1-|F(r/h)|\bigr).
\end{equation}
For each such $r$, there is a coupling $\pi$ of $\nu_+$ and $\nu_-$ such that $\pi\{(x,y):d_g(x,y)\leq 2r\}\geq 1-|F(r/h)|$.
If also $r\leq jh$, there is an averaging operator $L_r$ with the following properties. For each $x$, $L_r v(x)$ is the integral of $v$ against a probability measure supported in $\overline B_g(x,r)$, and
\begin{equation}\label{eq:cancellation-main}
  (K_r\varphi^\pm)\,dV_g
   =F(r/h)\mu_\pm+C_r\,dV_g,\qquad
  C_r\,dV_g=\bigl(1-F(r/h)\bigr)L_r^*\sigma .
\end{equation}
At $R=jh<r_g$ the three measures $K_R^*\mu_+$, $K_R^*\mu_-$ and $L_R^*\sigma$ agree. There is then a measure on $M^3$ with marginals $\mu_+$, $\mu_-$ and $\sigma$, supported on triples of diameter at most $2jh$. For every $\lambda>0$ there is such a measure with diameter bound $C_g h$.
\end{theorem}

To compare values of opposite signs we use the identity $F_\nu'(s)=-s F(s)/n$, which shows that the first minimum of $F_\nu$ is at $j$. This is also the global minimum, since the Bessel equation gives $(F_\nu^2+(F_\nu')^2)'=-2(n-1)(F_\nu')^2/s$. We write $\beta_n=-F_\nu(j)$.

\begin{samepage}
\begin{corollary}[Mass and opposite signs in a geodesic ball]\label{cor:balance}
If $R=jh<r_g$, then for each sign and every $y\in M$,
\begin{equation}\label{eq:ordinary-ball-balance}
 \frac{\int_{B_g(y,R)}\varphi^\pm\,dV_g}
      {\int_{B_g(y,R)}|\varphi|\,dV_g}=\frac12+O_g(h^2).
\end{equation}
For sufficiently small $h$ the same ball satisfies
\begin{equation}\label{eq:local-opposite-sharp}
 \max_{\overline B_g(y,R)}\varphi^\mp
       \ge (\beta_n-C_g h^2)\,\varphi^\pm(y).
\end{equation}
No larger dimensional constant works in place of $\beta_n$, even if the maximum on the left is taken over all of $M$.
\end{corollary}
\end{samepage}

The first assertion compares the $L^1$ masses locally, at every centre. Globally the two masses are equal because $\varphi$ has mean zero; Jakobson and Nadirashvili compared the global $L^p$ norms for $p\ge1$ \cite[Theorem~1, p.~398]{JN}. In dimensions at least three, Georgiev and the author showed that a sufficiently small ball of radius proportional to the wavelength, centred at a point of maximum of $|\varphi|$ on a nodal domain, has almost all its volume in that domain \cite[Theorem~1.6, Eq.~(2), p.~135]{GMbrownian}. In \eqref{eq:ordinary-ball-balance} the radius is $jh$, and the balance is in $L^1$ mass rather than volume.

Nadirashvili compared the global extrema at high frequency. His local comparison applies in a ball whose concentric half-ball contains a zero. At a sufficiently large multiple of the wavelength this gives a comparison around any centre \cite[Lemma~3 and proof of Theorem~2, pp.~263--264]{Nad}. Kr\"oger's global comparison gives $\|\varphi^\mp\|_\infty\ge\beta_n\|\varphi^\pm\|_\infty$ for every $n\ge2$ and $\lambda>0$ when $\operatorname{Ric}_g\ge0$ \cite[Theorem, p.~652, and Remarks~(5) and~(7), p.~653]{Kroger}. On a general fixed manifold, \eqref{eq:local-opposite-sharp} gives the leading constant $\beta_n$ up to an error of order $h^2$, with the maximum taken over the ball of radius $jh$ rather than over $M$. The first two constants are $\beta_2=0.402759\ldots$ and $\beta_3=0.217233\ldots$. Jakobson and Nadirashvili discussed the unequal extrema of the even zonal harmonics on $S^2$ and referred to Kr\"oger for their limiting ratio \cite[p.~405]{JN}.

The distance to the nodal measure has a sharper bound. Proposition~\ref{prop:nodal-radius} gives the first two terms, which are attained by zonal harmonics on the round sphere. Its proof uses the first Dirichlet eigenfunctions of small geodesic balls and gives a second proof of the upper bound in Theorem~\ref{new:endpoint} for small $h$. The identities in Theorem~\ref{thm:cancellation} and the comparison in Corollary~\ref{cor:balance} use the ball averages.

\subsection{Mass near and away from the nodal set}
Mass at distance $r$ from $Z$ must travel at least $r$ to reach the other sign. The author's nodal proof of the lower bound in all dimensions uses the following estimate. For $1\le p\le\infty$ the estimate appeared in \cite[Theorem~2.3 and Remark~3.13]{Muk}. Section~\ref{sec:oldproof} also covers $0<p<1$, where we write $\|v\|_{L^p(E)}=(\int_E|v|^p\,dV_g)^{1/p}$.

\begin{theorem}[Mass away from the nodal set]\label{old:nonconc}
For $0<p\leq\infty$ there are $a_p,c_p>0$, depending only on $p$ and $(M,g)$, such that, for $0<r\leq a_p h$ and either sign,
\[
 \|\varphi^\pm\|_{L^p(M\setminus T_r)}\geq c_p\|\varphi\|_{L^p(M)}.
\]
\end{theorem}

For the unnormalized measures, the lower bound in Theorem~\ref{new:endpoint} reads
\[
 W_1(\mu_+,\mu_-)\geq c_g h\|\varphi\|_1.
\]
We deduce this from Theorem~\ref{old:nonconc} by duality in Section~\ref{sec:thm_uncer}. The bound $\HH_g^{n-1}(Z)\geq c_g/h$ is due to Br\"uning on surfaces \cite[Satz~2]{Br} and to Logunov in higher dimensions \cite[Theorem~1.1]{Logunov}. Thus for $\|\varphi\|_1=1$ we have $W_1(\mu_+,\mu_-)\,\HH_g^{n-1}(Z)\geq c_g$ at high frequency. For uncertainty estimates relating transport to the size of a separating set for more general functions, see Sagiv and Steinerberger \cite[p.~3039]{SS}, Carroll, Massaneda and Ortega-Cerd\`a \cite[Theorem~2, p.~1160]{CMO} (the correction to the published proof appears in \cite[p.~1521]{CMOc}), and Cavalletti and Farinelli \cite[Theorem~1.1, p.~3]{CF}.

To find how much mass can travel a short distance, we also need the mass \emph{inside} $T_{sh}$ for small $s$. Its leading term on each side is determined by the normal derivative on $\Zreg$.

Write $T_r^\pm=T_r\cap\{\pm\varphi>0\}$. For $p>0$ and $f\in C^1(M)$ put
\[
 J_p(f)=\int_{\Zreg}f|\nabla\varphi|^p\,dS_g,\qquad
 J_p=J_p(1),\qquad \sigma_p=\min\{p,1\}.
\]

\begin{theorem}[The gradient on the nodal set]\label{thm:moments}
For every $p>0$ there are constants $c_p,C_p>0$, depending only on $(M,g)$ and $p$, such that
\begin{equation}\label{eq:moments}
 c_p\norm{\varphi}_p^p
 \leq h^{p+1}J_p\leq C_p\norm{\varphi}_p^p.
\end{equation}
For $p=1$,
\begin{equation}\label{eq:dong-total}
 J_1=\frac{\lambda}{2}\norm{\varphi}_1.
\end{equation}
\end{theorem}

At high frequency the cases $p=2$ and $p=3$ of the upper bound in \eqref{eq:moments} are due to Hezari and Sogge \cite[Eq.~(7), p.~1134]{HeS} and Hezari and Wang \cite[Eqs.~(12)--(13), p.~4]{HW}, respectively.

\begin{theorem}[Integrals on the two sides of a nodal tube]\label{thm:weighted}
Fix $p>0$. There are constants $s_p,C_p>0$, depending only on $(M,g)$ and $p$, such that the following estimates hold for $0<s<s_p$, any $f\in C^1(M)$ and either choice of sign.
\begin{align}
 \left|\int_{T_{sh}^\pm}f|\varphi|^p\,dV_g
 -\frac{s^{p+1}h^{p+1}}{p+1}J_p(f)\right|
 &\leq C_p s^{p+1+\sigma_p}
 \bigl(\norm f_\infty+h\norm{\nabla f}_\infty\bigr)
 \norm{\varphi}_p^p,\label{eq:weighted-mass}\\
 \left|\int_{T_{sh}^\pm}f(h|\nabla\varphi|)^p\,dV_g
 -s h^{p+1}J_p(f)\right|
 &\leq C_p s^{1+\sigma_p}
 \bigl(\norm f_\infty+h\norm{\nabla f}_\infty\bigr)
 \norm{\varphi}_p^p.\label{eq:weighted-gradient}
\end{align}
\end{theorem}

Jakobson and Nadirashvili \cite[Theorem~1, p.~398]{JN} showed that the global $L^p$ norms of the two signs are comparable for $p\ge1$, extending Nadirashvili's comparison of the extrema \cite[Theorem~2]{Nad}. They asked whether the ratio tends to one as $\lambda\to\infty$ for $1<p<\infty$ \cite[p.~405]{JN}. It can stay away from one on the standard flat torus in every dimension at least three \cite[Theorem~4, p.~2]{BEVW} and on the round two-sphere for $p\geq6$ \cite[Theorem~2]{BV}.

Inside a shrinking tube the ratio does tend to one. The same nodal integral gives the leading term on both sides in \eqref{eq:weighted-mass}, so with $f=1$ we get
\[
 \frac{\norm{\varphi^+}_{L^p(T_{sh})}}{\norm{\varphi^-}_{L^p(T_{sh})}}
  =1+O_{g,p}(s^{\sigma_p})\qquad(s\downarrow0)
\]
for every $0<p<\infty$, uniformly in the eigenfunction. Restrict either $|\varphi|\,dV_g$ or $h|\nabla\varphi|\,dV_g$ to $T_{sh}$ and divide by its mass there. As $\lambda\to\infty$ and $s\downarrow0$, these probability measures have the same weak limits as $|\varphi|\,dV_g/\|\varphi\|_1$ (Corollary~\ref{cor:weak}).

Adding the two signs in each estimate and cancelling $J_p$ gives
\begin{equation}\label{eq:intro-ratio}
 r^p\frac{\int_{T_r}|\nabla\varphi|^p\,dV_g}
 {\int_{T_r}|\varphi|^p\,dV_g}
 =p+1+O_p\bigl((r/h)^{\sigma_p}\bigr),
 \qquad 0<r<s_ph.
\end{equation}
Proposition~\ref{prop:common-moments} gives the limit of $\int_M C_r^p\,dV_g/\int_{T_r}|\varphi|^p\,dV_g$ as $r/h\downarrow0$. It depends only on $n$ and $p$ and equals $1/(2(n+2))$ for $p=1$.

The ordinary volume on each side also admits a lower bound. For $0<\delta<s_0h$, Proposition~\ref{prop:concentration-volume} gives $\min_\pm\Vol_g(T_\delta^\pm)\ge c_g(\delta/h)(\|\varphi\|_1/\|\varphi\|_2)^2$, with a corresponding estimate for other pairs of exponents. On the square flat torus this recovers the bound of order $\delta/h$ for each sign. There it also follows from the full-tube estimate \cite[Theorem~1.2, p.~1109]{JM}, since a translation exchanges the signs \cite[proof of Theorem~2.5]{MT}.

One can also start from a coupling. For a nonzero continuous real function $f$ of mean zero, let $V_\alpha(f)$ be the least volume carrying a fraction $\alpha$ of the integral of $|f|$. If a coupling of $f^+dV_g$ and $f^-dV_g$ assigns at least a fraction $\theta$ of its mass to pairs at distance at most $R$, each side of the $\delta$-tube around $\{f=0\}$ has volume at least $c_g\min\{1,\delta/R\}V_{\theta/2}(f)$ (Proposition~\ref{prop:transport-volume}).

\subsection{Sums of eigenfunctions}
Let $f\in L^2(M)$ be real, nonzero and of mean zero. Write
$\kappa_1<\kappa_2<\cdots$ for the distinct frequencies that occur in
$f$. Thus the projection of $f$ onto the eigenspace of eigenvalue
$\kappa_i^2$ is nonzero. Set $h_i=\kappa_i^{-1}$ and
$S_\ell=\sum_{i=1}^\ell h_i$, putting $h_i=0$ after the last
frequency if only finitely many occur. As before, write
$m_f=\norm f_1/2$, $\mu_\pm=f^\pm\,dV_g$ and $\nu_\pm=\mu_\pm/m_f$.

For a sum with $k$ distinct frequencies, Decio \cite[p.~935 and Theorem~1, p.~937]{Decio} put a zero within $C_{g,k}h_1$ of every point, provided the lowest frequency is sufficiently large in terms of $g$ and $k$. Steinerberger attributes to Decio the question whether the bound $C_g\sum_{i=1}^k h_i$ holds on every smooth closed manifold \cite[Section~2.3, p.~257]{StSign}. On the standard flat torus, Kozma and Oravecz gave a bound by summing over the Fourier support \cite[Theorem~1, pp.~447--448]{KO}. Steinerberger obtained a bound there in terms of the distinct wavelengths, as well as related bounds on round spheres and for global Helmholtz solutions in $\mathbb R^2$ and $\mathbb R^3$ \cite[Theorems~1--3, pp.~256--258]{StSign}.

\begin{theorem}[Transport for spectral sums]\label{thm:spectral-sums}
There is a constant $C_g$ and a coupling $\pi$ of $\mu_+$ and $\mu_-$
such that
\begin{equation}\label{eq:spectral-tail}
  \pi\{d_g(x,y)> C_g S_\ell\}\leq 2^{-\ell}m_f
                 \qquad(\ell\geq1).
\end{equation}
If $f$ has $k$ distinct frequencies, the left hand side vanishes for
$\ell\geq k$. The same coupling, divided by $m_f$, gives
\begin{equation}\label{eq:spectral-wasserstein}
 W_p(\nu_+,\nu_-)\leq C_g S_{\lceil p\rceil}
 \quad(1\leq p<\infty),\qquad
 W_\infty(\nu_+,\nu_-)\leq C_g\sum_i h_i
 \quad\text{if }\sum_i h_i<\infty.
\end{equation}
\end{theorem}

For a finite sum this puts a zero within $C_gS_k$ of every point, answering Decio's question for every $\kappa_1>0$. Indeed, mass near the centre of a zero-free ball has to leave the ball to reach the other sign. For an arbitrary sum, $S_\ell\leq\ell/\kappa_1$. The coupling therefore has an exponential tail at scale $\kappa_1^{-1}$ and gives $W_p(\nu_+,\nu_-)\le C_g p/\kappa_1$ for $1\le p<\infty$. For $p>1$ this removes the factor $\sqrt{\log(\kappa_1^2\|f\|_2/\|f\|_1)}$ in Steinerberger's high-frequency estimate \cite[Theorem~5, p.~318]{StFourier}. An example on the flat torus in Section~\ref{sec:spectral-sums} shows why the bound must grow linearly with $p$.

\section{Mass concentration away from the nodal set}\label{sec:oldproof}
Georgiev and the author used heat flow to bound the $L^1$ mass inside a nodal tube from below \cite[Theorem~1.1, Eq.~(2), p.~678]{GMheat}. On a closed surface at high frequency, their upper bound inside a sufficiently thin wavelength tube leaves a fixed fraction of the mass outside it for each $1\le p<\infty$ \cite[Theorem~1.2, Eq.~(3), pp.~678--679]{GMheat}. Here we use wavelength balls, which we shall need again for the nodal integrals in Section~\ref{sec:moments}.

Fix $0<p<\infty$. Every point lies within $R_0 h$ of a zero \cite[Lemma~8.1]{Logunov}. Take $R>8R_0$. For sufficiently large $\lambda$, choose a maximal $Rh$-separated set of centres $x_i$. The balls $B_i=B_g(x_i,Rh)$ cover $M$. The balls $4B_i=B_g(x_i,4Rh)$ have overlap bounded by a constant $N_4$. We call $B_i$ good if
\[
  \int_{4B_i}|\varphi|^p\,dV_g\leq K\int_{B_i}|\varphi|^p\,dV_g,
  \qquad K=2N_4 .
\]
The proof of \cite[Lemma~3, Eq.~(1.9), p.~779]{CM}, applied to $|\varphi|^p$ and $4B_i$, bounds the mass summed over the remaining balls by $K^{-1}N_4\|\varphi\|_p^p$. Thus
\begin{equation}\label{eq:good-mass}
 \sum_{B_i\ {\rm good}}\int_{B_i}|\varphi|^p\,dV_g\geq \tfrac12\|\varphi\|_p^p .
\end{equation}

On a good ball we measure lengths in units of $Rh$, so that $B_i$ becomes the unit ball. Write $B_a$ for the concentric ball of radius $a$ and normalize the rescaled function $v$ by $\int_{B_1}|v|^p=1$. Write $Z_v=\{v=0\}$. In normal coordinates the rescaled metrics $\gamma$ range over a smoothly precompact family, and
\[
 (\Delta_\gamma+R^2)v=0,\qquad \int_{B_4}|v|^p\leq K,
 \qquad Z_v\cap B_{1/8}\ne\varnothing .
\]
Following Jakobson and Nadirashvili \cite[proofs of Lemmas~5 and~8, pp.~402--404]{JN}, we claim that $B_{1/4}$ contains two balls of a fixed radius $b>0$, one with $v\ge b$ and one with $v\le-b$. Otherwise take a sequence for which no uniform $b$ works. The local maximum principle \cite[Theorem~9.20]{GT} on the harmonic lifts $\pm e^{Rt}v$ bounds the interior suprema for every $p>0$. After passing to a subsequence, the metrics and, by interior estimates, the functions $v_j$ converge smoothly on compact subsets of $B_2$.

The limit has $\int_{B_1}|v_\infty|^p=1$ and a zero in $\overline B_{1/8}$. By the strong minimum principle and unique continuation it takes both signs nearby. Choose two small closed balls there on which $v_\infty$ stays strictly positive and strictly negative. By uniform convergence and after decreasing $b$, the two balls satisfy $v_j\geq b$ and $v_j\leq-b$, respectively, for all sufficiently large $j$.

In the concentric half of either ball the distance to $Z_v$ is at least $b/2$, and the integral of $|v|^p$ is at least $c b^{n+p}$. Thus for each sign, the integral of $(\varphi^\pm)^p$ over $B_i\setminus T_{bRh/2}$ is at least $c b^{n+p}\int_{B_i}|\varphi|^p\,dV_g$. Summing over the good balls and using \eqref{eq:good-mass} and bounded overlap proves the estimate for large $\lambda$. For the finitely many remaining positive eigenvalues we use compactness in each normalized eigenspace. As above, two balls on which a limiting eigenfunction stays strictly positive and strictly negative also work for the sequence.

For $p=\infty$ we can argue directly at the extrema. By Nadirashvili's theorem \cite[Theorem~2]{Nad}, both have magnitude at least $c_g\|\varphi\|_\infty$ at high frequency. For the remaining eigenvalues the same conclusion follows by compactness in each eigenspace, since every nonzero eigenfunction of positive eigenvalue takes both signs. The gradient bound $\|\nabla\varphi\|_\infty\le C_g h^{-1}\|\varphi\|_\infty$ then puts each extremum at distance at least $c'_g h$ from $Z$.

\begin{remark}
For a familiar example on $S^2$, take $\varphi=\operatorname{Re}(x_1+ix_2)^l$. In latitude $u$ and longitude $\theta$ it is $(\cos u)^l\cos(l\theta)$. The band $|u|\leq l^{-1/2}$ carries a fixed fraction of the integral of $|\varphi|^p$ for each fixed $p>0$. Indeed $\int_{-\pi/2}^{\pi/2}(\cos u)^{lp+1}\,du\asymp_p l^{-1/2}$, and the integrand is bounded below in this band. There one may also restrict to $\pm\cos(l\theta)\geq1/2$ and retain a fixed fraction for either sign. These points are at distance at least $c/l$ from the nodal meridians. Here $h=(l(l+1))^{-1/2}$.
\end{remark}

\section{Wasserstein distance and the two signs}\label{sec:thm_uncer}
\subsection{The lower bound}
Apply Kantorovich--Rubinstein duality to the normalized measures and multiply by their common mass \cite[Theorem~1.14, p.~34]{Vi}. This gives $W_1(\mu_+,\mu_-)=\sup_{\operatorname{Lip}(f)\leq1}\int f\,d(\mu_+-\mu_-)$. We take the signed distance to $Z$, truncated at $a_1h$,
\[
 f(x)=\operatorname{sgn}(\varphi(x))\min\{d_g(x,Z),a_1h\}.
\]
On each sign set $f$ is $1$-Lipschitz. A minimizing geodesic between points of opposite sign meets $Z$, so the sum of the two distances to $Z$ is at most its length. Thus $f$ is an admissible test and the dual integral is at least $a_1h\int_{M\setminus T_{a_1h}}|\varphi|\,dV_g$. Theorem~\ref{old:nonconc} with $p=1$ gives the lower bound in Theorem~\ref{new:endpoint}.

\subsection{Ball averages and cancellation}\label{sec:cancellation}
In Euclidean space the average of a Helmholtz eigenfunction over $B(x,r)$ is $F(r/h)\varphi(x)$ \cite[Eq.~(1), p.~281]{Volchkov}. In particular its integral over $B(x,jh)$ vanishes. On $M$ the corresponding operator is
\begin{equation}\label{eq:ball-multiplier}
 F(s)=2^{n/2}\Gamma(n/2+1)s^{-n/2}J_{n/2}(s),\qquad
 F(0)=1,\qquad K_r=F(r\sqrt{-\Delta_g}).
\end{equation}
Here $\omega_n$ will denote the Euclidean unit ball volume.

\begin{lemma}\label{lem:ball-kernel}
There is $r_g>0$, smaller than the injectivity radius, such that for $0<r<r_g$ the operator $K_r=F(r\sqrt{-\Delta_g})$ preserves constants and has a symmetric kernel $k_r$ satisfying
\begin{equation}\label{eq:ball-kernel}
 k_r(x,y)=\frac{\mathbf 1_{\{d_g(x,y)<r\}}}{\omega_n r^n}
                 \bigl(1+O_g(r^2)\bigr)
 \qquad\text{for almost all }(x,y)\in M\times M.
\end{equation}
The operator $K_r$ is a positive contraction on $C(M)$. For every $x\in M$, the probability measure defining $K_r v(x)$ has density between $1-C_gr^2$ and $1+C_gr^2$ relative to $(\omega_n r^n)^{-1}\mathbf 1_{B_g(x,r)}\,dV_g$.
\end{lemma}

\begin{proof}
For a radial density $b$ supported in $|w|\le r$, let $H(s)=\int b(w)\cos(sw_1)\,dw$. Finite propagation makes the kernel $k_H$ of $H(\sqrt{-\Delta_g})$ vanish when $d=d_g(x,y)>r$. Inside the ball, for smooth $b$, small $r$ and a fixed integer $Q>n+3$, the estimate of Gariboldi and Gigante \cite[Theorem~5, p.~792]{GG} reads in our convention
\[
 |k_H(x,y)-a_0(x,y)b(d)|
 \le C_g\sum_{\ell=1}^Q\int_d^\infty v^{2\ell-1}|b(v)|\,dv
       +C_g\int_0^\infty v^{2Q+1}|b(v)|\,dv .
\]
Here $b(v)$ denotes the radial profile. The constant is independent of $b$ and $r$: it comes from bounds on the Hadamard coefficients and the remainder for the wave kernel. Rescaling the metric and the radius together gives their volume-one convention. The leading Hadamard coefficient $a_0$ is smooth and satisfies $a_0(x,x)=1$ \cite[Theorem~2, p.~788]{GG}. Since $k_H$ is symmetric, we may replace $a_0(x,y)$ by its symmetric part. Its derivatives in the two variables then agree on the diagonal, and their sum is the derivative of $a_0(x,x)=1$. Both vanish there, giving $a_0(x,y)=1+O_g(d^2)$.

We shall use two densities $b$ that are smooth inside the ball. To approximate either one, take $b_m=\chi_m b$, where the smooth radial cutoffs $0\le\chi_m\le1$ are supported in $|w|<r$ and increase to one there. The radial extension of the corresponding multiplier $H_m$ is Schwartz. By Fourier inversion its transform is $(2\pi)^n b_m(2\pi\,\cdot)$ in the convention of \cite[p.~785]{GG}, with support in the ball of radius $r/(2\pi)$. The estimate above therefore applies to $b_m$.

Any uniform bound for the errors passes to the limit. Indeed, $b_m\to b$ in $L^1(\RR^n)$, so $H_m\to H$ uniformly and the operators converge in $L^2$ operator norm. Geodesic polar coordinates give $a_0 b_m(d)\to a_0 b(d)$ in $L^1(M\times M)$. The error kernels have a weak-star convergent subsequence in $L^\infty(M\times M)$. Testing against products of smooth functions identifies the sum of the limits with the kernel of $H(\sqrt{-\Delta_g})$.

For the normalized ball indicator $b$, each radial error integral is $O_g(r^{2-n})$ for $r<1$. Since $a_0=1+O_g(d^2)$, we obtain \eqref{eq:ball-kernel} almost everywhere. The set $\{(x,y):d_g(x,y)=r\}$ has zero product volume.

For $r_g$ small enough the kernel is nonnegative and is positive almost everywhere on $d_g(x,y)<r$. Also, $F(0)=1$ gives $K_r1=1$. The operator $K_r$ maps smooth functions to smooth functions, hence $K_r$ extends by density to a contraction on $C(M)$. For continuous $v\ge0$, both $K_rv(x)$ and the ball integral of $v$ are continuous in $x$, extending the density bounds to every centre.
\end{proof}

\begin{proof}[Proof of Theorem~\ref{thm:cancellation}]
Put $H(t)=F(t/h)$. The two averages have sum $K_r|\varphi|$ and difference $H(r)\varphi$. Their minimum is half the sum minus half the absolute difference. Integrating and using symmetry and $K_r1=1$, we obtain
\[
 C_r=\tfrac12\bigl(K_r|\varphi|-|H(r)|\,|\varphi|\bigr),
 \qquad
 \int_M C_r\,dV_g=m(1-|H(r)|).
\]
The smaller average is positive precisely when the open ball contains both signs. A ball meeting $Z$ contains open sets of both signs by the strong minimum principle and unique continuation. Elsewhere the sign is constant. This gives $\{C_r>0\}=T_r$.

Sogge and Zelditch \cite[Proposition~1, Eq.~(4), p.~26]{SZ}, following a related integral formula of Dong \cite[Theorem~2.1, Eq.~(1), p.~495]{Dong}, proved that
\begin{equation}\label{eq:dong-weighted}
 2\int_{\Zreg}v|\nabla\varphi|\,dS_g
 =\int_M(\Delta_g v+\lambda v)|\varphi|\,dV_g,\qquad v\in C^2(M).
\end{equation}
Since $2\varphi^\pm=|\varphi|\pm\varphi$ and $(\Delta_g+\lambda)\varphi=0$, the identity for either sign reads
\begin{equation}\label{eq:one-sided-flux}
 \int_M (v+h^2\Delta_g v)\,d\mu_\pm=\int_{\Zreg}v\,d\sigma,
 \qquad v\in C^\infty(M).
\end{equation}
In particular $\sigma(M)=m$, by taking $v=1$.

To express the common part in terms of $\sigma$, fix a smooth test function $v$ and, for either sign, put
\[
 U(t)=\int_M K_t v\,d\mu_\pm-H(t)\int_M v\,d\mu_\pm .
\]
The Bessel equation for $F$, applied to each term in the spectral expansion of $v$, gives
$\partial_t^2 K_t v+(n+1)t^{-1}\partial_t K_tv=\Delta_g K_t v$.
We may differentiate term by term because the spectral coefficients of $v$ decay rapidly and all derivatives of $F$ are bounded by its Fourier integral. For $H$ the right hand side of the scalar equation is $-h^{-2}H$. Substituting \eqref{eq:one-sided-flux} with test function $K_tv$ gives
\[
 U''+\frac{n+1}{t}U'+h^{-2}U=h^{-2}\int_{\Zreg}K_t v\,d\sigma,
 \qquad U(0)=U'(0)=0.
\]

We already know one solution of the homogeneous equation, namely $H$. Multiplying by $t^{n+1}H(t)$ and integrating gives
\begin{equation}\label{eq:flux-wronskian}
 t^{n+1}\bigl(H(t)U'(t)-H'(t)U(t)\bigr)
   = h^{-2}\int_0^t u^{n+1}H(u)\!\int_{\Zreg}K_u v\,d\sigma\,du .
\end{equation}
For $r<jh$ we have $H>0$ on $[0,r]$. Dividing by $t^{n+1}H(t)^2$ gives $(U/H)'$ on the left. Integrating and reversing the order gives
\begin{align}
 U(r)&=\int_0^r b_r(u)\!\int_{\Zreg}K_u v\,d\sigma\,du,\label{eq:flux-averaging}\\
 b_r(u)&=h^{-2}u^{n+1}H(u)H(r)
                 \int_u^r\frac{dt}{t^{n+1}H(t)^2},\qquad 0<u<r.\notag
\end{align}
The weight $b_r(u)$ is positive and is $O_{r,h}(u)$ near zero. At $R=jh$ we cannot divide by $H(R)$. Instead, in \eqref{eq:flux-wronskian} set $t=R$ and use $H(R)=0$ and $H'(R)<0$. We get \eqref{eq:flux-averaging} at $R$ with the positive weight
\[
 b_R(u)=\frac{u^{n+1}H(u)}{-h^2 R^{n+1}H'(R)},\qquad 0<u<R.
\]
To find the total weight, take $v=1$ in \eqref{eq:flux-averaging}. Both $\mu_\pm$ and $\sigma$ have mass $m$, so $\int_0^r b_r(u)\,du=1-H(r)$. For $r<r_g$ and $0<r\le jh$ define
\[
 L_r v=\frac{1}{1-H(r)}\int_0^r b_r(u)K_u v\,du.
\]
This average preserves constants, extends by uniform approximation to $C(M)$, and moves mass by at most $r$ since each $K_u$ moves it by at most $u$. By symmetry, \eqref{eq:flux-averaging} becomes
\[
 (K_r\varphi^\pm)\,dV_g=H(r)\mu_\pm+(1-H(r))L_r^*\sigma.
\]
Since $H(r)\ge0$ and $\mu_+,\mu_-$ are mutually singular, the part common to these two measures is precisely $(1-H(r))L_r^*\sigma$.

Fix $0<r<r_g$. Where $C_r(z)>0$, the measures
\[
 \alpha_z^\pm(dx)=
 \frac{k_r(z,x)\varphi^\pm(x)}{K_r\varphi^\pm(z)}\,dV_g(x)
\]
are probabilities supported in $\overline B_g(z,r)$. Integrating $\alpha_z^+\otimes\alpha_z^-$ against $C_r(z)\,dV_g(z)$ gives a measure on $M\times M$. Its marginals are dominated by $\mu_\pm$ since $C_r\le K_r\varphi^\pm$, and its mass is $m(1-|H(r)|)$. It is supported on pairs at distance at most $2r$. Couple the remaining marginals arbitrarily and divide the resulting measure by $m$.

At $R=jh<r_g$ the common measure is also $L_R^*\sigma$. Disintegrate the measure obtained by applying the kernel of $L_R$ to $\sigma$ with respect to its image $L_R^*\sigma$. This gives a conditional probability $\alpha_z^0$ on the original nodal points, supported in $\overline B_g(z,R)$. Integrating $\alpha_z^+\otimes\alpha_z^-\otimes\alpha_z^0$ against $L_R^*\sigma$ gives the three stated marginals and the diameter bound $2R$. If $jh\ge r_g$, the product of the marginals divided by $m^2$ suffices, since $\diam(M)\le j\diam(M)h/r_g$.
\end{proof}

\begin{proof}[Proof of \eqref{eq:ordinary-ball-balance}]
Write $\langle v\rangle_B$ for the ordinary ball average. Since $\Vol_g(B_g(x,r))=\omega_n r^n(1+O_g(r^2))$, Lemma~\ref{lem:ball-kernel} and $K_r\varphi=F(r/h)\varphi$ give
\[
 \left|\langle\varphi\rangle_{B_g(x,r)}-F(r/h)\varphi(x)\right|
 \leq C_g r^2\langle|\varphi|\rangle_{B_g(x,r)},\qquad 0<r<r_g.
\]
The denominator in the ratio of signed to unsigned mass is positive by unique continuation. When $F(r/h)=0$, the estimate just proved makes this ratio $O_g(r^2)$, uniformly in the centre and in the choice of zero of $F$. Each sign therefore accounts for a fraction $1/2+O_g(r^2)$ of the mass, in particular at $r=jh$.
\end{proof}

At the first Bessel zero one can also express the ordinary mass of each sign in terms of $\sigma$ inside the ball. Set $B(v)=\int_v^1 wF(jw)\,dw$ for $0\le v\le1$, and $B(v)=0$ for $v>1$. For $R=jh<r_g$ and every $x\in M$,
\begin{equation}\label{eq:ball-nodal-identity}
 \int_{B_g(x,R)}\varphi^\pm\,dV_g
   =\frac{j}{|F'(j)|}\bigl(1+O_g(R^2)\bigr)
      \int_{\Zreg}B\bigl(d_g(x,z)/R\bigr)\,d\sigma(z).
\end{equation}

\begin{proof}[Proof of \eqref{eq:ball-nodal-identity}]
We use the radial test function $B(d_g(x,\cdot)/R)$ in \eqref{eq:one-sided-flux}. The Bessel equation gives $B''+(n-1)v^{-1}B'+j^2B=-jF'(j)>0$ for $0<v<1$. The function $B$ is smooth and even near zero, and both $B$ and $B'$ vanish at $1$. With $d=d_g(x,y)<R$, the formula $\Delta_g d=(n-1)/d+O_g(d)$ now gives
\[
 (\Delta_g+h^{-2})B(d/R)= R^{-2}\bigl(-jF'(j)+O_g(R^2)\bigr).
\]
Extension by zero introduces no boundary term since $B'(1)=0$. Smooth approximations to this $C^{1,1}$ test function converge uniformly with bounded Hessians converging off the sphere. The sphere is $\mu_\pm$-null, so \eqref{eq:one-sided-flux} applies by dominated convergence. Substitution of $h=R/j$ gives \eqref{eq:ball-nodal-identity} for small $r_g$.
\end{proof}

\subsection{Consequences of local cancellation}
Projecting the measure in Theorem~\ref{thm:cancellation} onto its first two coordinates and dividing by $m$ gives the coupling in Theorem~\ref{new:endpoint}. In terms of sets, write $A_R=\{y:d_g(y,A)\leq R\}$ for a Borel set $A\subset M$. Then $\mu_+(A)\leq\mu_-(A_R)$ for $R=2jh<2r_g$, and the same holds with the signs interchanged. Projecting onto the first and third or the second and third coordinates gives $W_\infty(\nu_\pm,\sigma/m)\leq C_g h$. For a lower bound, the distance from $x$ to any point of $Z$ is at least $d_g(x,Z)$. The one-sided estimate of Theorem~\ref{old:nonconc} at $p=1$ yields
\[
 c_g h\leq W_p(\nu_\pm,\sigma/m)\leq C_g h\qquad(1\leq p\leq\infty).
\]

The ball average gives an explicit comparison of opposite-sign values with no error term. Nadirashvili uses a zero in the inner half-ball and obtains a comparison at an unspecified multiple of the wavelength \cite[Lemma~3 and proof of Theorem~2, pp.~263--264]{Nad}. Here we use that $F(s_*)<0$ at its first minimum $s_*=j_{n/2+1,1}$. When $s_*h<r_g$, the value $K_{s_*h}\varphi=F(s_*)\varphi$ is a probability average. It follows that
\[
 \max_{\overline B_g(y,s_*h)}\varphi^\mp\ \geq\ |F(s_*)|\varphi^\pm(y).
\]

For comparison, the heat average gives an identity at every time. Let $P_t=e^{t\Delta_g}$ and $d\omega=|\nabla\varphi|\,dS_g|_{\Zreg}$. For $t>0$,
\[
 \min(P_t\varphi^+,P_t\varphi^-)
   =\int_0^t e^{-\lambda(t-u)}P_u\omega\,du,
 \qquad \int_M\min(P_t\varphi^+,P_t\varphi^-)\,dV_g
   =m(1-e^{-\lambda t}).
\]
In fact $(\Delta_g+\lambda)\varphi^\pm=\omega$ distributionally by \eqref{eq:one-sided-flux}. Integrating the heat equation gives $P_t\varphi^\pm=e^{-\lambda t}\varphi^\pm+\int_0^t e^{-\lambda(t-u)}P_u\omega\,du$. At each point one of the two initial terms vanishes. Integration gives the second formula.

\subsection{Distance to the nodal measure}\label{sec:radius}
\begin{proposition}\label{prop:nodal-radius}
For sufficiently small $h$,
\begin{equation}\label{eq:sharp-nodal-radius}
 W_\infty(\nu_\pm,\sigma/m)
 \le j_{(n-2)/2,1}h-\frac{j_{(n-2)/2,1}}{12}(\min_M\mathrm{Scal}_g)h^3+C_gh^4.
\end{equation}
On the round sphere the first two terms agree with the expansion of $W_\infty(\nu_+,\sigma/m)$ along zonal harmonics positive at the pole. Reversing the sign gives the corresponding example for $\nu_-$.
\end{proposition}

\begin{proof}
We test \eqref{eq:one-sided-flux} with the first Dirichlet eigenfunctions of small geodesic balls. Write $a=j_{\nu,1}$ and $S=\min_M\mathrm{Scal}_g$. The first Dirichlet eigenvalue of $B_g(x,R)$ satisfies
\[
 \lambda_1^D(B_g(x,R))=\frac{a^2}{R^2}-\frac{\mathrm{Scal}_g(x)}6+O_g(R).
\]
For the curvature term we use the small-ball expansion in Druet \cite[Step~1, pp.~1166--1167]{Druet}. For the uniform remainder, pull back to the unit ball by $z\mapsto\exp_x(Rz)$ and use the rescaled metric $g_{ij}(Rz)$ in normal coordinates. These metrics depend smoothly on the centre, the orthonormal frame and $R$, including at $R=0$. The first eigenpair on the fixed ball depends smoothly on these parameters by simplicity and elliptic regularity. Compactness of the orthonormal frame bundle then gives a uniform Taylor remainder $O_g(R^3)$ for the scaled eigenvalue after its quadratic term. Take $R=ah-aSh^3/12+C_0h^4$, with $C_0$ large enough that $\lambda_1^D(B_g(x,R))\le h^{-2}$ at every centre for small $h$.

The first eigenfunction on the scaled ball is the Euclidean first Dirichlet eigenfunction plus a correction of order $R^2$ and a remainder $O_g(R^3)$ in $C^2$ on a smaller ball. The linear term vanishes in normal coordinates, and the quadratic term is even in $z$ because the quadratic part of the metric is even and the normalized eigenpair is simple. The scaled eigenfunction therefore has gradient $O_g(R^3)$ at the origin. The Euclidean first Dirichlet eigenfunction has a unique maximum at zero with negative definite Hessian. Uniform convergence on the ball and $C^2$ convergence near zero give a unique maximum for the perturbed eigenfunction, at distance $O_g(R^4)$ from $x$ in the original metric. The maximum is independent of the frame and depends continuously on $x$. Moving along the short minimizing geodesic from $x$ to the maximum gives a homotopy to the identity, so this map has degree one modulo two and is onto.

Extend each first eigenfunction by zero and normalize its maximum to one. Inside the ball, $(\Delta_g+\lambda)u=(\lambda-\lambda_1^D)u\ge0$. On the boundary the extra term is minus the outward normal derivative times surface measure, which is nonnegative. Hence $(\Delta_g+\lambda)u\ge0$ in the sense of distributions. The functions are uniformly Lipschitz for fixed $h$. Let $A$ be compact, let $\delta>0$, and take the supremum $U$ over the functions whose maxima are in $A_\delta$. Then $0\le U\le1$, $U=1$ on the open $\delta$-neighbourhood of $A$, and $U$ vanishes outside $A_{R+C_1h^4+\delta}$ for a fixed $C_1$. The distributional inequality is preserved by this supremum. Indeed, with $P_t=e^{t\Delta_g}$, we have
\[
 \frac{d}{dt}(e^{\lambda t}P_tu)=e^{\lambda t}P_t((\Delta_g+\lambda)u)\ge0 .
\]
Thus $u\le e^{\lambda t}P_tu\le e^{\lambda t}P_tU$. Taking the supremum gives $U\le e^{\lambda t}P_tU$. Integrate against a smooth nonnegative function, move $P_t$ to that function by symmetry, and divide the difference by $t$. Letting $t$ tend to zero gives $(\Delta_g+\lambda)U\ge0$.

The latter distribution is a nonnegative measure and equals $\lambda\,dV_g$ on the open $\delta$-neighbourhood of $A$. Testing \eqref{eq:one-sided-flux} with $U$ gives
\[
 \mu_\pm(A)\le h^2\int_M\varphi^\pm\,d((\Delta_g+\lambda)U)
 =\int_M U\,d\sigma\le\sigma(A_{R+C_1h^4+\delta}).
\]
For the middle equality, approximate $U$ and $\varphi^\pm$ separately by smooth functions, uniformly and in $H^1$. Both integrals then equal
\[
 -h^2\int_M\langle\nabla U,\nabla\varphi^\pm\rangle_g\,dV_g
 +\int_M U\varphi^\pm\,dV_g .
\]
Let $\delta$ decrease to zero and apply Strassen's theorem \cite[Corollary~1.28, p.~44]{Vi}. The inequality for compact sets suffices by inner regularity of $\mu_\pm$.

On the round sphere the first zero of a zonal harmonic that is positive at the pole bounds a cap on which the harmonic is the first Dirichlet eigenfunction. Every neighbourhood of the pole contains positive mass, so the distance from the pole to that first zero is a lower bound for $W_\infty(\nu_+,\sigma/m)$. Applying the small-ball expansion to the cap gives the matching lower bound with the same first two terms.
\end{proof}

The triangle inequality for $W_\infty$ now gives \eqref{eq:improved-radius}.

\subsection{Comparison of extrema}

For \eqref{eq:local-opposite-sharp} we give more weight to the outer part of the ball. Take $0<\epsilon\leq1$, put $\alpha=\nu+\epsilon$, and use the probability density
\[
 b_{r,\epsilon}(w)=c_{n,\epsilon}r^{-n}
       (1-|w|^2/r^2)^{\epsilon-1}\mathbf 1_{\{|w|<r\}},\qquad
 c_{n,\epsilon}=\frac{\Gamma(n/2+\epsilon)}{\pi^{n/2}\Gamma(\epsilon)} .
\]
Integrating first over the sections $w_1=t$ and then using the integral formula for $J_\alpha$ gives $F_\alpha(rs)=2^\alpha\Gamma(\alpha+1)(rs)^{-\alpha}J_\alpha(rs)$, with value one at $s=0$. We need the corresponding operator on $M$ to have a nonnegative kernel. For small $r$ the leading term inside the ball is at least $c_n\epsilon r^{-n}$, since $c_{n,\epsilon}\asymp_n\epsilon$ and $a_0$ in Lemma~\ref{lem:ball-kernel} is near one. The error is at most $C_g r^{2-n}$, uniformly in $\epsilon$, so it suffices to take $r^2\leq c_g\epsilon$. Indeed the first radial error integral is bounded by
\[
 c_{n,\epsilon}r^{2-n}\int_0^1 t(1-t^2)^{\epsilon-1}\,dt
   =\frac{c_{n,\epsilon}}{2\epsilon}r^{2-n}.
\]
The higher radial moments have additional powers of $r$. The same bounds hold for the smooth cutoffs of $b_{r,\epsilon}$, so the approximation in the proof of Lemma~\ref{lem:ball-kernel} gives the limiting kernel with error $O_g(r^{2-n})$. The kernel is supported in $d\le r$, and $F_\alpha(0)=1$ gives mass one. For $r^2\le c_g\epsilon$ the leading term exceeds the error inside the ball.

Take $r=jh$ and $\epsilon=A_gh^2$ with $A_g$ large enough that $r^2\le c_g\epsilon$. Then smooth dependence on the order gives
\[
 F_{\nu+\epsilon}(j)=F_\nu(j)+O_g(h^2)=-\beta_n+O_g(h^2).
\]
For $\varphi(y)>0$ the average $F_{\nu+\epsilon}(j)\varphi(y)$ is at least the minimum of $\varphi$ on $\overline B_g(y,jh)$. This gives \eqref{eq:local-opposite-sharp} first for almost every $y$, then for every $y$ by continuity of $\varphi$ and the ball minimum. Apply the same argument to $-\varphi$ for the other sign.

When $\operatorname{Ric}_g\ge0$, one can take the curvature parameter in Kr\"oger's equation to be zero, so the regular solution is $F_\nu$ \cite[Eq.~(1), p.~652]{Kroger}.

On $S^2$, Kr\"oger showed that $P_{2k}(\cos\theta)$ maximizes $\|\varphi^+\|_\infty/\|\varphi^-\|_\infty$ in degree $2k$ \cite[Example, p.~653]{Kroger}. Beiner and Verreault give the limiting calculation explicitly \cite[Section~3.4, pp.~390--391]{BV}. For sharpness on $S^n$, take the zonal harmonic $Y_k$ of even degree $k$, with $Y_k(0)=1$ and $\theta$ the distance from the north pole. Then $h_k=(k(k+n-1))^{-1/2}$ and $Y_k(\theta)=P_k^{(\nu,\nu)}(\cos\theta)/P_k^{(\nu,\nu)}(1)$.

Since $kh_k\to1$, the Mehler--Heine formula \cite[Theorem~8.1.1, p.~192]{Szego}, normalized at zero, gives $Y_k(h_k s)\to F_\nu(s)$ uniformly on compact subsets of $\mathbb C$. The derivatives converge by Cauchy's formula. The same holds for the derivative divided by $s$, since both profiles are even. The limiting derivative is strictly negative on $(0,j)$, with a simple zero at $j$ and a nonzero limit after division by $s$ at zero. Thus the first zero and the first negative minimum occur at $h_k(j_{\nu,1}+o(1))$ and $h_k(j+o(1))$, respectively. On the northern hemisphere the radial equation gives
\[
 \frac{d}{d\theta}\bigl(Y_k^2+h_k^2(Y_k')^2\bigr)
 =-2(n-1)h_k^2\cot\theta\,(Y_k')^2\le0 .
\]
Successive extrema therefore decrease in absolute value. Since $Y_k(\pi-\theta)=Y_k(\theta)$, the positive maximum is one and the first negative minimum is the global minimum. Hence $\|Y_k^-\|_\infty/\|Y_k^+\|_\infty\to\beta_n$.

Every eigenfunction satisfies
\[
 W_\infty(\nu_+,\nu_-)\ge \max_{x\in M} d_g(x,Z).
\]
Indeed, if $d_g(x,Z)>R$, a neighbourhood of $x$ carries mass of one sign at distance greater than $R$ from the other. For $Y_k$ the distance at the pole is $h_k(j_{\nu,1}+o(1))$. What is the least $D_n$ such that $W_\infty/h\le D_n+o_g(1)$ at high frequency on every fixed smooth closed $n$-manifold? The example and \eqref{eq:improved-radius} give $j_{\nu,1}\le D_n\le2j_{\nu,1}$.

For a zonal harmonic the transport distance can be read off from the polar angle alone. Let $G_\pm$ denote the distribution functions of this angle under the two sign measures $\nu_\pm$ of $Y_k$. The densities are proportional to $Y_k^\pm(t)\sin^{n-1}t$ on $[0,\pi]$. Write $Q_\pm(u)=\inf\{t:G_\pm(t)\ge u\}$ for the quantiles. For every $k\ge1$ we have
\begin{align}
 W_\infty(\nu_+,\nu_-)&=\|Q_+-Q_-\|_{L^\infty(0,1)}
      &&\text{on }S^n,\label{eq:zonal-quantiles}\\
 W_\infty(\nu_+,\nu_-)&=\frac{\pi}{k+1},\qquad
 \frac{W_\infty}{h_k}=\pi\sqrt{1-\frac{1}{(k+1)^2}}\longrightarrow\pi=j_{1/2,1}
      &&\text{on }S^3.\label{eq:zonal-three-sphere}
\end{align}
Indeed, the polar angle is $1$-Lipschitz. A coupling supported on pairs at distance at most $R$ therefore gives $G_\pm(t)\le G_\mp(t+R)$ and hence $Q_\mp(u)\le Q_\pm(u)+R$. To obtain equality in \eqref{eq:zonal-quantiles}, take $u$ uniformly in $(0,1)$ and independently choose a uniform direction at the pole. Put both points on the resulting meridian, at angles $Q_+(u)$ and $Q_-(u)$. They have laws $\nu_+$ and $\nu_-$, and their distance is exactly $|Q_+(u)-Q_-(u)|$.

On the unit $S^3$ put $K=k+1$ and $R=\pi/K$. The radial equation gives $v''+K^2v=0$ for $v(t)=Y_k(t)\sin t$, with $v(0)=0$ and $v'(0)=1$. Hence $Y_k(t)=\sin(Kt)/(K\sin t)$. For $0<t<\pi-R$,
\[
 \frac{d}{dt}\bigl(G_\mp(t+R)-G_\pm(t)\bigr)
   =c_k(\sin(Kt))^\pm\bigl(\sin(t+R)-\sin t\bigr),
\]
with the same normalizing constant $c_k>0$ for both signs. The derivative is nonnegative up to $(\pi-R)/2$ and nonpositive thereafter. At both endpoints the difference is nonnegative, so $G_\pm(t)\le G_\mp(t+R)$ for all $t$. The quantiles therefore differ by at most $R$. The first nodal sphere is at angle $R$, so $W_\infty\ge R$ as well.

Does $W_\infty$ equal the distance from the pole to the first nodal sphere in every dimension $n\ge2$ and degree $k\ge1$? For $k=1$ the answer follows directly from \eqref{eq:zonal-quantiles}. Here $Y_1(t)=\cos t$, so $Q_+(u)=\arcsin(u^{1/n})$ and $Q_-(u)=\pi-Q_+(1-u)$. Since $u^{1/n}\ge\sqrt u$ and $\arcsin\sqrt u+\arcsin\sqrt{1-u}=\pi/2$, the displacement is at most $\pi/2$. It tends to $\pi/2$ as $u\downarrow0$.

On $S^2$ one can also test this question near the pole as $k\to\infty$, where the limiting profile is $J_0(|x|)$. In this planar model the discs $B(0,b)$, $0<b\le a=j_{0,1}$, give no larger lower bound. Put $P(t)=tJ_1(t)$. Since $P'=tJ_0$ and $2a$ lies below the second positive zero of $J_0$, the negative mass in $B(0,a+b)$ minus the positive mass in $B(0,b)$ is $2\pi E(b)$, where $E(b)=P(a)-P(a+b)-P(b)$. From $P''+P=J_0$ and $E(0)=E'(0)=0$ we get
\[
 E(b)=\int_0^b\sin(b-t)\bigl[P(a)-J_0(a+t)-J_0(t)\bigr]\,dt\ge0.
\]
Here $a<\pi$, $J_0(a+t)\le0$ and $J_0(t)\le1$. Also $P(a)>1$, as one sees by integrating $J_0(t)>1-t^2/4$ against $t\,dt$ on $(0,2)\subset(0,a)$.

For $\cos x$ on the circle, pair $x\in[0,\pi/2]$ with $y=\pi-\arcsin(1-\sin x)$. The negative mass from $\pi/2$ to $y$ equals the positive mass $\sin x$ from $0$ to $x$. Reflection completes the coupling, and differentiation gives $\max(y-x)=2\pi/3$ at $x=\pi/6$. Conversely, the arcs $[-\pi/6,\pi/6]$ about the positive maxima carry unit mass per period. Their closed $R$-neighbourhood contains less than unit negative mass for $R<2\pi/3$. Thus $W_\infty=2\pi h/3$ for $\cos(kx)$ by rescaling, strictly between $j_{-1/2,1}h=\pi h/2$ and $\pi h$. Coupling along coordinate circles and projecting back give the same value on the standard product torus.

For $\varphi(x,y)=\cos x\cos y$ on the square flat torus, $W_\infty$ strictly exceeds the greatest distance $R=\pi/2$ to the nodal set. A disc $B(0,b)$ about the positive maximum carries mass of order $b^2$ as $b\downarrow0$. In $B(0,R+b)$ the negative set consists of four circular segments of height $b$ and width $O(\sqrt b)$, on which $|\varphi|\le b$. Their total mass is $O(b^{5/2})$. For small fixed $b$ the strict mass inequality persists on increasing $R$ slightly, which gives the claim.

\section{Mass on the two sides of a shrinking tube}\label{sec:weighted-statements}
At a regular zero $z$, let $\nu$ be the unit normal towards the positive side. Taylor's formula gives $\varphi(\exp_z(t\nu))=t|\nabla\varphi(z)|+O(t^2)$. Integrating the linear term gives the leading mass on either side. The smaller of the ball averages of its positive and negative parts gives the leading term for $C_r$. In both cases we need an error estimate that also covers zeros where the gradient is small or vanishes.

For the lower bound in Theorem~\ref{thm:moments} we use the good balls of Section~\ref{sec:oldproof}, following \cite[Lemma~3, p.~779]{CM}. A compactness argument, as in \cite[proof of Lemma~5]{JN}, gives a regular piece of the nodal set on each such ball.

\subsection{Local nodal measure and derivative estimates}\label{sec:local}

Write $Z_v=\{v=0\}$ and $\T_\delta(v)=\{x:d(x,Z_v)<\delta\}$. We shall work in a chart containing the concentric $\gamma$-balls $B_R$ that occur below. The metrics $\gamma$ range over a smoothly precompact family of uniformly elliptic metrics with a common positive lower bound on the injectivity radius. Constants may depend on this family. On rescaling sufficiently small normal coordinate balls, the product metrics used in the proof are uniformly close to the Euclidean metric in $C^1$. We can thus use the same chart radii for all of them, as in \cite[Remark~3.1 and proof of Lemma~4.1]{Lupper}.

\begin{lemma}\label{lem:local}
Fix $\kappa_*>0$. There is a number $L>8$ with the following property. Suppose that $v\not\equiv0$ satisfies
\begin{equation}\label{eq:local-pde}
 \Delta_\gamma v+\kappa^2v=0\quad\hbox{on }B_L,
 \qquad 0\leq\kappa\leq\kappa_*.
\end{equation}
Fix $p>0$ and put
\[
 a=\int_{B_4}|v|^p\,dV_\gamma,
 \qquad b=\int_{B_L}|v|^p\,dV_\gamma,
 \qquad D=\log(b/a).
\]
There are constants $A>0$ and $C_p>0$ such that
\begin{align}
 \norm{v}_{C^2(B_3)}^p&\leq C_pa,\label{eq:c2}\\
 \HH_\gamma^{n-1}(Z_v\cap B_2)&\leq C_p(1+D)^A,\label{eq:area-growth}\\
 \norm{v}_{C^2(B_3)}^p\HH_\gamma^{n-1}(Z_v\cap B_2)
 &\leq C_pb.\label{eq:product}
\end{align}
There are also constants $c,\rho_0>0$ such that for $z\in Z_v\cap B_2$ and $0<\rho<\rho_0$,
\begin{equation}\label{eq:lower-density}
 \HH_\gamma^{n-1}(Z_v\cap B_\gamma(z,\rho))
 \geq c\rho^{n-1},
\end{equation}
provided that this ball is contained in $B_3$. Consequently, for all sufficiently small $\delta>0$,
\begin{equation}\label{eq:tube-volume}
 \Vol_\gamma\bigl(\T_\delta(v)\cap B_1\bigr)
 \leq C\delta\,\HH_\gamma^{n-1}(Z_v\cap B_2).
\end{equation}
The radius $L$ and the constants are independent of $v$ and of $\kappa\in[0,\kappa_*]$.
\end{lemma}

\begin{proof}
The lift $U(x,t)=e^{\kappa t}v(x)$ is harmonic, and
$\int_{B_4\times(-1,1)}|U|^p\,dV_\gamma\,dt\le C_pa$.
Apply the local maximum principle \cite[Theorem~9.20, (9.48)]{GT} to $U$ and to $-U$ on interior coordinate balls in the product. It gives $\sup_{B_{7/2}}|v|\le C_pa^{1/p}$ for every $p>0$. Interior Schauder estimates now give \eqref{eq:c2}.

For the nodal area bound we need control on small balls with different centres. The same interior estimate gives $|U|\le C_p b^{1/p}$ on $B_5\times(-2,2)$. Write $V=U/(C_p b^{1/p})$. By the definition of $a$, some $(x_0,0)\in B_4\times\{0\}$ satisfies $|V(x_0,0)|\ge c_p e^{-D/p}$. To pass from this point to the other centres, let $M_x(r)=\sup_{B(x,r)}|V|$ in the product metric. In normal coordinates about $x$, \cite[Lemma~1.3, (5)]{Lupper} with $t=4$ and $\varepsilon=1/4$ gives
\[
 M_x(2s)\le C M_x(s)^{1/3}M_x(4s)^{2/3}
\]
for a fixed sufficiently small $s>0$. Join $(x_0,0)$ to $x\in B_2\times(-1,1)$ by a chain in $B_4\times[-1,1]$, with successive centres at distance at most $s$. Each $s$-ball lies in the next $2s$-ball. Since $M_x(4s)\le1$ along the chain, the displayed inequality carries a lower bound from one $s$-ball to the next. A bounded number of steps gives $M_x(s)\ge e^{-C_p(1+D)}$, also in a fixed small neighbourhood of $B_2\times(-1,1)$.

We now have $\log_2(M_x(2s)/M_x(s))\le C_p(1+D)$ at every centre in this neighbourhood. The concentric estimate \cite[Lemma~5.2, Eq.~(73), p.~20]{GM} gives the same bound at every radius $r\le s/2$. Cover $B_2\times(-1,1)$ by finitely many sufficiently small coordinate cubes. Coordinate and metric balls are comparable there, so a fixed number of doublings bounds the cube indices of \cite[Section~4]{Lupper} by $C_p(1+D)$.

Logunov's estimate \cite[Theorem~6.1, (11)]{Lupper} in dimension $n+1$ gives $\HH^n(Z_U\cap(B_2\times(-1,1)))\le C_p(1+D)^A$. As $Z_U=Z_v\times\RR$, the product formula gives \eqref{eq:area-growth}. Multiplication by \eqref{eq:c2} proves \eqref{eq:product}, since $a(1+D)^A=be^{-D}(1+D)^A\le Cb$.

For $z\in Z_v\cap B_2$, the scaled lower bound \cite[Section~7, Theorem~7.1]{Logunov} gives $\HH^n(Z_U\cap B((z,0),\rho))\ge c\rho^n$. This ball lies in $B_\gamma(z,\rho)\times(-\rho,\rho)$. By \cite[Theorem~(1.10)]{HS}, the regular zeros of $v$ form a smooth hypersurface and the singular zeros have zero hypersurface measure. The nodal measure in the cylinder is therefore $2\rho\,\HH_\gamma^{n-1}(Z_v\cap B_\gamma(z,\rho))$, which proves \eqref{eq:lower-density}.

Finally we adapt the covering argument of \cite[proof of (5), Eqs.~(49)--(50), pp.~13--14]{GM}, using balls centred at zeros. Choose a maximal $\delta$-separated set of zeros in $B_{1+\delta}$. The $\delta/3$-balls about them are disjoint, and by \eqref{eq:lower-density} each contains nodal measure at least $c\delta^{n-1}$. The balls of radius $2\delta$ about the same zeros cover $\T_\delta(v)\cap B_1$. Comparing their total volume with the nodal measure in $B_2$ gives \eqref{eq:tube-volume}.
\end{proof}

\begin{remark}\label{rem:scaling}
Take $\rho=\eta h$, with $\eta>0$ so small that $L\eta\lambda_1^{-1/2}$ is below a fixed coordinate radius, where $\lambda_1$ is the first positive eigenvalue. In the coordinates $x=\exp_{x_j}(\rho y)$, divide the pullback metric by $\rho^2$. This gives \eqref{eq:local-pde} with $\kappa=\eta$. The metrics form a smoothly precompact family and tend to the Euclidean metric as $\lambda\to\infty$, so Lemma~\ref{lem:local} applies uniformly. Choose a cover by $\rho$-balls with bounded overlap of their $L$-fold enlargements, and a subordinate partition of unity $\chi_j$ with $|\nabla\chi_j|\le C/\rho$.
\end{remark}

\subsection{The gradient on the nodal set}\label{sec:moments}

On each rescaled ball of Remark~\ref{rem:scaling}, bound the gradient by its supremum in \eqref{eq:product}. In the original metric,
\[
 \int_{Z\cap B(x_j,\rho)}|\nabla\varphi|^p\,dS
 \leq C_p\rho^{-p-1}
 \int_{B(x_j,L\rho)}|\varphi|^p\,dV_g.
\]
Since $\rho=\eta h$, summation and bounded overlap prove the upper bound in \eqref{eq:moments}.

For the lower bound we need a piece of the nodal set on each good ball where the gradient is bounded below after normalization.

\begin{lemma}\label{lem:patch}
Fix $p>0$, $K>0$, and $\kappa\geq0$. Let $v$ solve
$\Delta_\gamma v+\kappa^2v=0$ on $B_4$, with metrics in a smoothly precompact family as above. Suppose that
\begin{equation}\label{eq:patch-assumptions}
 \int_{B_1}|v|^p\,dV_\gamma=1,
 \qquad \int_{B_4}|v|^p\,dV_\gamma\leq K,
 \qquad Z_v\cap\overline{B_{1/4}}\ne\varnothing.
\end{equation}
There is a constant $c>0$, independent of $v$, such that
\begin{equation}\label{eq:patch-result}
 \int_{Z_v\cap B_1}|\nabla v|^p\,d\HH_\gamma^{n-1}\geq c.
\end{equation}
\end{lemma}

\begin{proof}
Suppose the nodal integral in \eqref{eq:patch-result} tends to zero along a sequence. As in Section~\ref{sec:oldproof}, the bounds in \eqref{eq:patch-assumptions} give interior estimates even when $p<1$. After passage to a subsequence, the metrics and the solutions converge smoothly on $B_2$. The limit still satisfies $\int_{B_1}|v_\infty|^p=1$. It also has a zero in $\overline B_{1/4}$, since the chosen zeros have a convergent subsequence there.

The limit changes sign in every neighbourhood of its zero, by the strong minimum principle and unique continuation. In a small coordinate ball inside $B_{1/2}$, join two points of opposite sign by a segment. The endpoints still have opposite signs for all sufficiently small parallel translates of the segment. Each segment meets the zero set. The projection of the zero set onto the transverse hyperplane therefore contains a disk, so this portion of the zero set has positive hypersurface measure. The singular set has zero hypersurface measure by \cite[Theorem~(1.10)]{HS}, leaving a regular zero in $B_{1/2}$. There the implicit function theorem gives nodal graphs for the sequence over a fixed smaller disk. Their areas and the gradients on them are bounded below, contradicting convergence of the integral to zero.
\end{proof}

We use the good balls and the normalization in Section~\ref{sec:oldproof}, with the same $R$ and $K$. For each such ball Lemma~\ref{lem:patch}, after scaling back, gives
\[
 \int_{Z\cap B_i}|\nabla\varphi|^p\,dS_g\geq c_p(Rh)^{-p-1}\int_{B_i}|\varphi|^p\,dV_g .
\]
Summing, using bounded overlap and \eqref{eq:good-mass}, gives the lower bound for large $\lambda$.
For the finitely many smaller positive eigenvalues, the set $\int_M|\varphi|^p\,dV_g=1$ in each eigenspace is compact in every smooth norm, also for $p<1$. Every limit has a regular zero, so the same graph argument bounds $J_p$ below uniformly. Decreasing $c_p$ gives \eqref{eq:moments}.

For $p=1$, \eqref{eq:dong-total} follows by taking $v=1$ in \eqref{eq:dong-weighted}.

At the lower endpoint, fix the eigenfunction and let $p\downarrow0$. The middle quantity in \eqref{eq:moments} tends to $h\HH_g^{n-1}(Z)$, since the singular zeros have zero hypersurface measure. Similarly $\|\varphi\|_p^p\to\Vol_g(M)$ since $Z$ has zero volume. Lower and upper constants uniform down to $p=0$ would thus give the corresponding bounds in Yau's conjecture. At the other endpoint there are constants depending only on $(M,g)$ such that
\[
 c_g\|\varphi\|_\infty\leq h\max_{Z}|\nabla\varphi|\leq C_g\|\varphi\|_\infty .
\]
The upper bound is the interior derivative estimate. For the lower bound suppose $\|\varphi_j\|_\infty=1$ and $h_j\max_{Z_j}|\nabla\varphi_j|\to0$. If $h_j\to0$, rescale by $h_j$ at points where $|\varphi_j|=1$. The rescaled metrics converge smoothly to the Euclidean metric on compact sets. Interior estimates and a diagonal subsequence give an entire limit $\Delta v+v=0$ with $|v(0)|=1$. Wavelength density gives a zero in a fixed rescaled ball and hence a zero of the limit. The sign change and projection argument in Lemma~\ref{lem:patch} gives a regular zero of $v$. The implicit function theorem then gives zeros of the rescaled functions with gradients bounded below, a contradiction. For bounded eigenvalues use compactness in one eigenspace.
\subsection{Weighted integrals over nodal tubes}\label{sec:weighted}

Where the normal derivative is large enough, we integrate along each half of a normal segment. Near the remaining zeros, Taylor's formula bounds the integrand itself.

\begin{lemma}\label{lem:weighted-local}
Let $v$ satisfy the hypotheses of Lemma~\ref{lem:local}, and let $\psi\in C^1_c(B_1)$. Write
\[
 H=\norm{v}_{C^2(B_3)},\qquad
 A_Z=\HH_\gamma^{n-1}(Z_v\cap B_2),\qquad
 U_\delta^\pm=\{x:\dist_\gamma(x,Z_v)<\delta,\ \pm v(x)>0\}.
\]
Fix $p>0$ and put $\sigma_p=\min\{p,1\}$. For all sufficiently small $\delta>0$ and for either sign, we have
\begin{align}
 \left|\int_{U_\delta^\pm}\psi|v|^p\,dV_\gamma
 -\frac{\delta^{p+1}}{p+1}
 \int_{(Z_v)_{\rm reg}}\psi|\nabla v|^p\,d\HH_\gamma^{n-1}\right|
 &\leq C_p\delta^{p+1+\sigma_p}\norm{\psi}_{C^1}H^pA_Z,
 \label{eq:local-weighted-mass}\\
 \left|\int_{U_\delta^\pm}\psi|\nabla v|^p\,dV_\gamma
 -\delta\int_{(Z_v)_{\rm reg}}\psi|\nabla v|^p\,d\HH_\gamma^{n-1}\right|
 &\leq C_p\delta^{1+\sigma_p}\norm{\psi}_{C^1}H^pA_Z.
 \label{eq:local-weighted-gradient}
\end{align}
Moreover, $H^pA_Z\leq C_p\int_{B_L}|v|^p\,dV_\gamma$.
\end{lemma}

\begin{proof}
Write $a(z)=|\nabla v(z)|$. In $H=\|v\|_{C^2(B_3)}$ we take derivatives with respect to $\gamma$. This gives a norm uniformly equivalent to the coordinate norm. For a large geometric constant $K_0$ let
\[
 G_\delta=\{z\in Z_v\cap B_2:a(z)>K_0\delta H\}.
\]
Orient the normal towards $\{v>0\}$ and set
\[
 \nu(z)=\frac{\nabla v(z)}{a(z)},\qquad
 F(z,t)=\exp_z(t\nu(z)),\qquad |t|<\delta.
\]
We claim that $z$ is the unique nearest zero to $F(z,t)$. If $d_\gamma(w,F(z,t))\le |t|$ for a zero $w$, then $d_\gamma(w,z)\le 2\delta$. Put $\xi=\exp_z^{-1}(w)$. Since $v(z)=v(w)=0$, Taylor's theorem gives
\begin{equation}\label{eq:normal-taylor}
 |\ip{\nu(z)}{\xi}|\leq\frac{CH}{a(z)}|\xi|^2.
\end{equation}
At $z$ the squared distance to $F(z,t)$ has gradient $-2t\nu(z)$. Along the geodesic from $z$ to $w$ the distance to $F(z,t)$ is at most $3\delta$, so the curvature bound gives a lower bound $2(1-C\delta^2)\gamma$ for the Hessian of the squared distance. It follows that
\begin{align*}
 d_\gamma(F(z,t),w)^2
 &\geq t^2-2t\ip{\nu(z)}{\xi}+(1-C\delta^2)|\xi|^2\\
 &\geq t^2+\left(1-C\delta^2-\frac{C|t|H}{a(z)}\right)|\xi|^2.
\end{align*}
For $K_0$ large and then $\delta$ small the last coefficient is positive. Thus $w=z$ and $F$ is injective.

The second fundamental form of the regular nodal set at $z$ is bounded by $CH/a(z)$. If $\mathcal J$ denotes the volume Jacobian of $F$, the Jacobi field equation along each normal gives
\begin{equation}\label{eq:normal-jacobian}
 \mathcal J(z,t)=1+O\left(\frac{H|t|}{a(z)}\right).
\end{equation}
Here $H|t|/a(z)<1/K_0$ and the curvature error $O(t^2)$ is $O(H|t|/a(z))$ since $a(z)\le CH$. Taylor's theorem and parallel transport along the normal give
\begin{equation}\label{eq:normal-values}
 v(F(z,t))=a(z)t+O(Ht^2),\qquad
 \bigl||\nabla v(F(z,t))|-a(z)\bigr|\leq CH|t|.
\end{equation}
Increasing $K_0$ if necessary, we see that $v(F(z,t))$ has the sign of $t$ for $t\ne0$.

To take the $p$th power, observe from \eqref{eq:normal-values} that both $|v(F(z,t))|/(a(z)|t|)$ for $t\ne0$ and $|\nabla v(F(z,t))|/a(z)$ stay in a fixed compact subinterval of $(0,\infty)$. On this interval $x^p$ has bounded derivative for every $p>0$. Hence
\begin{align*}
 \bigl||v(F(z,t))|^p-a(z)^p|t|^p\bigr|
 &\leq C_pHa(z)^{p-1}|t|^{p+1},\\
 \bigl||\nabla v(F(z,t))|^p-a(z)^p\bigr|
 &\leq C_pHa(z)^{p-1}|t|.
\end{align*}
To integrate along either half of the segment, we also use \eqref{eq:normal-jacobian} and $|\psi(F(z,t))-\psi(z)|\leq C\norm{\psi}_{C^1}|t|$. The errors in the integrals of $\psi|v|^p$ and $\psi|\nabla v|^p$ along that half are bounded by
\[
 C_p\norm{\psi}_{C^1}\delta^{p+2}Ha(z)^{p-1}
 \quad\hbox{and}\quad
 C_p\norm{\psi}_{C^1}\delta^2Ha(z)^{p-1},
\]
respectively. Integration of $|t|^p$ and of $1$ gives the two leading terms
\[
 \frac{\delta^{p+1}}{p+1}\psi(z)a(z)^p
 \quad\hbox{and}\quad
 \delta\psi(z)a(z)^p.
\]
To integrate the errors with respect to $z$, observe that on $G_\delta$ we have
\begin{equation}\label{eq:good-power-bound}
   Ha(z)^{p-1}\leq C_p H^p\delta^{\sigma_p-1}.
\end{equation}
For $p\ge1$ we used $a(z)\le CH$. For $p<1$ the exponent $p-1$ is negative, so we use $a(z)>K_0\delta H$. Divide $G_\delta$ into disjoint measurable pieces in regular nodal charts. Since $F$ is injective we may change variables on their union. Integration with \eqref{eq:good-power-bound} gives the errors in \eqref{eq:local-weighted-mass} and \eqref{eq:local-weighted-gradient} on these segments.

It remains to consider $x\in U_\delta^\pm\cap\supp\psi$ outside these segments. A nearest zero $z$ lies in $B_2$. If $z$ belonged to $G_\delta$, the minimizing geodesic from $z$ to $x$ would be normal to the nodal set, and we would already have counted $x$. Thus $a(z)\le K_0\delta H$, including when $z$ is singular. Taylor's theorem gives
\[
 |v(x)|\leq C\delta^2H,\qquad
 |\nabla v(x)|\leq C\delta H.
\]
By \eqref{eq:tube-volume}, these points occupy volume at most $C\delta A_Z$. Their contributions to the two integrals are therefore at most
\[
 C_p\norm{\psi}_\infty\delta^{2p+1}H^pA_Z
 \quad\hbox{and}\quad
 C_p\norm{\psi}_\infty\delta^{p+1}H^pA_Z.
\]
The regular zeros omitted from the leading integrals also satisfy $a(z)\leq K_0\delta H$, so their two contributions obey the same bounds. Since $2p+1\geq p+1+\sigma_p$ and $p+1\geq1+\sigma_p$, this proves both estimates. The bound for $H^pA_Z$ is \eqref{eq:product}.
\end{proof}

\begin{proof}[Proof of Theorem~\ref{thm:weighted}]
We use the cover of Remark~\ref{rem:scaling}, with $\rho=\eta h$, and its partition of unity $\{\chi_j\}$. In the chart about $x_j$ write
\[
 v_j(y)=\varphi(\exp_{x_j}(\rho y)),\qquad
 \delta=s/\eta,\qquad
 \psi_j=(f\chi_j)\circ\exp_{x_j}(\rho\,\cdot).
\]
The cutoff bound gives $\norm{\psi_j}_{C^1}\le C(\norm f_\infty+h\norm{\nabla f}_\infty)$. Every zero within $sh$ of $\supp\chi_j$ lies in the enlarged chart, so distances to such zeros can be computed there. The nodal integral rescales as
\[
 \int_{Z_{v_j}}\psi_j|\nabla v_j|^p\,d\HH^{n-1}
 =\rho^{p+1-n}\int_Zf\chi_j|\nabla\varphi|^p\,dS.
\]
On returning to the original metric, the volume element contributes $\rho^n$ in \eqref{eq:local-weighted-mass}, and the factor for \eqref{eq:local-weighted-gradient} is $\rho^n(h/\rho)^p$. We get the claimed leading terms with $f\chi_j$ in place of $f$. By Lemma~\ref{lem:weighted-local} the errors are bounded by
\[
 C_p(\norm f_\infty+h\norm{\nabla f}_\infty)
 \int_{LB_j}|\varphi|^p\,dV_g
\]
multiplied by $s^{p+1+\sigma_p}$ and $s^{1+\sigma_p}$, respectively. Since $\sum_j\chi_j=1$ and the balls $LB_j$ have bounded overlap, summation gives the two estimates on $M$.
\end{proof}

At $f=1$, \eqref{eq:moments} gives relative errors $O_p(s^{\sigma_p})$. Thus
\[
 \int_{T_{sh}^\pm}|\varphi|^p\,dV_g\asymp_p s^{p+1}\norm{\varphi}_p^p,
 \qquad
 \int_{T_{sh}^\pm}(h|\nabla\varphi|)^p\,dV_g\asymp_p s\norm{\varphi}_p^p .
\]
Choose $s=s(g,p)>0$ small and keep it fixed. Then $\int_M(\varphi^\pm)^p\,dV_g$ is at least a fixed fraction of $\int_M|\varphi|^p\,dV_g$ for each sign. This gives the global comparison of Jakobson and Nadirashvili \cite[Theorem~1]{JN} for finite $p\ge1$, and the same comparison for $0<p<1$.

For each integrand the ratio of its integrals over $T_{sh}^+$ and $T_{sh}^-$ is $1+O_p(s^{\sigma_p})$. For $p=2$, \eqref{eq:intro-ratio} gives the limit $3$ for the ratio of the gradient and mass integrals over the whole tube after multiplication by $r^2$, or by $s^2$ if the gradient is written as $h\nabla\varphi$.

For $p=1$, \eqref{eq:dong-total} gives
\begin{align}
 \frac{1}{\norm{\varphi}_1}\int_{T_{sh}}|\varphi|\,dV
 &=\frac{s^2}{2}+O(s^3),\label{eq:quadratic}\\
 \frac{1}{\norm{\varphi}_1}\int_{T_{sh}}h|\nabla\varphi|\,dV
 &=s+O(s^2).\label{eq:linear-gradient}
\end{align}
On either side of the tube the leading terms are $s^2/4$ and $s/2$, respectively.

\begin{corollary}\label{cor:weak}
Let $d\nu_\varphi=|\varphi|\,dV/\norm{\varphi}_1$. For $f\in C^2(M)$ and all sufficiently small $s>0$, we have
\begin{align}
 &\left|\frac{2}{s^2\norm{\varphi}_1}\int_{T_{sh}}f|\varphi|\,dV
       -\int_Mf\,d\nu_\varphi\right|\notag\\
 &\quad+\left|\frac{1}{s\norm{\varphi}_1}\int_{T_{sh}}fh|\nabla\varphi|\,dV
       -\int_Mf\,d\nu_\varphi\right|\notag\\
 &\hspace{15mm}\leq Cs\bigl(\norm{f}_\infty+h\norm{\nabla f}_\infty\bigr)
                  +Ch^2\norm{\Delta_gf}_\infty.
 \label{eq:weak-tube-measures}
\end{align}
Suppose $\lambda_j\to\infty$ and $s_j\to0$. If $\nu_{\varphi_j}$ converges weakly to a probability measure $\nu$, then each of the measures
\[
 \mathbf 1_{T_{s_jh_j}}|\varphi_j|\,dV
 \quad\hbox{or}\quad
 \mathbf 1_{T_{s_jh_j}}h_j|\nabla\varphi_j|\,dV
\]
converges weakly to $\nu$ after normalization to total mass one.
\end{corollary}

\begin{proof}
For a test function $f$, \eqref{eq:dong-weighted} reads
\[
 2h^2\int_Zf|\nabla\varphi|\,dS
 =\int_Mf|\varphi|\,dV+h^2\int_M(\Delta_gf)|\varphi|\,dV.
\]
Substitution in Theorem~\ref{thm:weighted} gives \eqref{eq:weak-tube-measures}. Each rescaled measure has mass $1+O(s)$ by \eqref{eq:quadratic} and \eqref{eq:linear-gradient}. Dividing by this mass gives the claimed weak limit for smooth tests, and then for continuous tests by uniform approximation.
\end{proof}

\subsection{Integrating powers of the smaller average}

The constant in the expansion for $C_r$ comes from averaging a linear function over a Euclidean ball. For $0\leq u\leq1$ and $p>0$, let
\[
 b_n(u)=\frac{1}{\omega_n}\int_{|w|<1}(w_1-u)_+\,dw,\qquad
 A_{n,p}=2\int_0^1 b_n(u)^p\,du .
\]
For the linear function $y\mapsto ay_1$, $a>0$, consider a ball of radius $r$ centred at $te_1$ with $|t|\le r$. After rescaling to the unit ball, the smaller of the averages of its positive and negative parts is $ar\,b_n(|t|/r)$.

\begin{proposition}\label{prop:common-moments}
For every $p>0$ there exist $s_p,C_p>0$, depending only on $(M,g)$ and $p$, such that for $f\in C^1(M)$ and $r=sh$ with $0<s<s_p$,
\begin{align}
 \left|\int_M f C_r^p\,dV_g-A_{n,p}r^{p+1}J_p(f)\right|
 &\leq C_p s^{p+1+\sigma_p}
  \bigl(\norm f_\infty+h\norm{\nabla f}_\infty\bigr)
  \norm{\varphi}_p^p .\label{eq:common-moments}
\end{align}
In particular,
\begin{equation}\label{eq:common-ratio}
 \frac{\displaystyle\int_M C_r^p\,dV_g}
      {\displaystyle\int_{T_r}|\varphi|^p\,dV_g}
 =\frac{p+1}{2}A_{n,p}+O_p(s^{\sigma_p}).
\end{equation}
At $p=1$ the limit is $1/(2(n+2))$.
\end{proposition}

\begin{proof}
Use the charts of Theorem~\ref{thm:weighted} and write $v=v_j$, $\psi=\psi_j$, $\delta=r/\rho$ and $c=C_r$ in these coordinates. We keep the notation $H,A_Z,G_\delta$ of Lemma~\ref{lem:weighted-local}. Since $h\le\lambda_1^{-1/2}$, we can choose $s_p$ so that $r<r_g$ for every eigenfunction.

At $x=F(z,t)$ with $z\in G_\delta$, $|t|<\delta$, choose a frame with first vector the parallel translate of $\nu(z)$. Writing $a=|\nabla v(z)|$, Taylor expansion from $z$ to $x$ and then from $x$ gives
\[
 v(\exp_x(\delta w))=a(t+\delta w_1)+O(H\delta^2),\qquad |w|<1.
\]
The averaging density on $|w|<1$ is $\omega_n^{-1}(1+O_g(\delta^2))$ by the normal volume density and \eqref{eq:ball-kernel}, since $r=\rho\delta$ and $\rho$ is bounded. Comparison with the linear function gives
\[
 c(x)=a\delta b_n(|t|/\delta)+O(H\delta^2).
\]
Indeed the maps $t\mapsto\max\{\pm t,0\}$ are $1$-Lipschitz, taking the minimum preserves the larger error bound, and $a\le CH$ controls the density error.

Both quantities in this comparison are at most $Ca\delta$, because $a>K_0\delta H$. For $p\ge1$, the mean value theorem bounds the difference of their $p$th powers by $C_pHa^{p-1}\delta^{p+1}$. For $p<1$ we use H\"older continuity of $x^p$ at zero. This gives $C_pH^p\delta^{2p}$ even at the ends of the segment, where $b_n$ vanishes.

Replacing $\mathcal J$ by one and $\psi(x)$ by $\psi(z)$ introduces errors bounded in the integrand by $C(a\delta)^p(\norm\psi_\infty H\delta/a+\delta\norm{\nabla\psi}_\infty)$, by \eqref{eq:normal-jacobian}. By \eqref{eq:good-power-bound} all errors integrated over the segments are at most $C_p\norm{\psi}_{C^1}\delta^{p+1+\sigma_p}H^pA_Z$. The leading term integrated from $-\delta$ to $\delta$ is $A_{n,p}\delta^{p+1}\psi(z)a^p$.

If $x$ lies elsewhere in the tube on $\supp\psi$, a nearest zero $z$ has $a(z)\le K_0\delta H$. Taylor's theorem then gives $|v|\le CH\delta^2$ throughout $B(x,\delta)\subset B(z,2\delta)$ and hence $c(x)\le CH\delta^2$. By \eqref{eq:tube-volume} these points contribute at most $C_p\norm{\psi}_\infty\delta^{2p+1}H^pA_Z$. The same bound holds for the omitted nodal integral. Outside the tube $C_r=0$ by Theorem~\ref{thm:cancellation}.

Rescaling by $\rho^n$ and summing gives $A_{n,p}r^{p+1}J_p(f)$, with error \eqref{eq:common-moments} by \eqref{eq:product} and bounded overlap. For $f=1$, divide by the sum of \eqref{eq:weighted-mass} over both signs, using \eqref{eq:moments} for the relative error. At $p=1$, integration in $u$ and rotational invariance give
\[
 \int_0^1 b_n(u)\,du
 =\frac{1}{2\omega_n}\int_{|w|<1}(w_1)_+^2\,dw
 =\frac{1}{4(n+2)} .
\]
\end{proof}

\section{Short-distance transport and one-sided tube volumes}\label{sec:transport}

\begin{proof}[Proof of Theorem~\ref{thm:transport}]
Let $r=sh$. Consider a pair of points $x,y$ with $\varphi(x)>0>\varphi(y)$.
A minimizing geodesic from $x$ to $y$ passes through a zero, so
\begin{equation}\label{eq:transport-crossing}
 \dist(x,Z)+\dist(y,Z)\leq d_g(x,y).
\end{equation}
If $d_g(x,y)\le r$, at least one endpoint is within $r/2$ of $Z$. Thus for any coupling $\pi\in\Pi(\mu_+,\mu_-)$,
\[
 \pi\{d_g(x,y)\leq r\}
 \leq\mu_+\{\dist(x,Z)\leq r/2\}
     +\mu_-\{\dist(x,Z)\leq r/2\}.
\]
The right hand side is the unsigned mass in the closed $r/2$-tube. Letting slightly larger open tubes decrease to it in \eqref{eq:quadratic} gives
\begin{equation}\label{eq:transport-upper}
 \pi\{d_g(x,y)\leq r\}
 \leq m\left(\frac{s^2}{4}+Cs^3\right).
\end{equation}

For the reverse bound we construct a measure on pairs of points on opposite sides of the same normal segment. To first order the density at distance $t$ is $|\nabla\varphi(z)|t\,dt$. To bound the error, take the charts $B_j$ of radius $\rho=\eta h$ in Remark~\ref{rem:scaling}. If $v_j$ is the rescaled eigenfunction, then $K_j=C\rho^{-2}\norm{v_j}_{C^2(B_3)}$ bounds $|\nabla^2\varphi|$ on the enlarged chart, and $|\nabla\varphi|\le C\rho K_j$. Rescaling \eqref{eq:product} with $p=1$ gives
\[
 K_j\HH_g^{n-1}(Z\cap B_j)
 \leq C\rho^{-3}\int_{LB_j}|\varphi|\dV.
\]
Since the balls $LB_j$ have bounded overlap, summing gives
\begin{equation}\label{eq:transport-hessian-sum}
 \sum_jK_j\HH_g^{n-1}(Z\cap B_j)
 \leq Ch^{-3}\norm{\varphi}_1.
\end{equation}

We count each regular zero only once, assigning it to the first ball $B_j$ that contains it. This gives a disjoint Borel partition $\Zreg=\bigcup_j P_j$. With $a(z)=|\nabla\varphi(z)|$ and a large fixed $K_0$, let
\[
 G_j=\{z\in P_j:a(z)>K_0rK_j\},
 \qquad G=\bigcup_jG_j.
\]
At $z\in G$ orient the normal towards the positive side by setting
$\nu(z)=\nabla\varphi(z)/a(z)$, and write
\[
 F(z,t)=\exp_z(t\nu(z)),\qquad |t|<r.
\]
Any competing nearest zero would lie within $2r$ of $z$, hence in the enlarged chart. The uniqueness proved in Lemma~\ref{lem:weighted-local} therefore makes the segments disjoint even across different $P_j$.

Write $dV_g=\mathcal J(z,t)\,dS(z)\,dt$. Here $\mathcal J$ is the normal Jacobian. By \eqref{eq:normal-jacobian} and \eqref{eq:normal-values}, $\varphi(F(z,t))$ has the sign of $t$ and
\begin{equation}\label{eq:transport-normal-density}
 \varphi^\pm(F(z,\pm t))\mathcal J(z,\pm t)
 =a(z)t+O(K_jt^2),\qquad z\in G_j,\quad 0<t<r.
\end{equation}
On both sides of $z$ use a density slightly smaller than the linear term:
\[
 q_z(t)=(a(z)-C_0rK_j)t,
 \qquad z\in G_j,\quad 0<t<b:=r/\sqrt2.
\]
Choose $C_0$ large enough that $q_z$ is bounded above by both
densities in \eqref{eq:transport-normal-density}, and $K_0>2C_0$ so
that $q_z>0$. Map $(z,t)$ to $(F(z,t),F(z,-t'))$, where
\[
 t'=\sqrt{b^2-t^2}.
\]
The substitution preserves $q_z(t)\,dt$ because $t\,dt=-t'\,dt'$.
The two points in each pair have distance
\[
 d_g(F(z,t),F(z,-t'))\leq t+t'\leq\sqrt2b=r.
\]
Let $\pi_0$ be the image of $q_z(t)\,dt\,dS(z)$ on $G\times(0,b)$ under this map. The density bound \eqref{eq:transport-normal-density} and the change from $t$ to $t'$ show that in each chart its marginals are at most $\mu_\pm$. Since the normal segments are disjoint, these bounds hold on $M$.

Since $\int_0^b t\,dt=r^2/4$, the mass of $\pi_0$ is
\[
 M_r=\frac{r^2}{4}\sum_j\int_{G_j}
       (a(z)-C_0rK_j)\,dS(z).
\]
The zeros in $P_j\setminus G_j$ have $a(z)\leq K_0rK_j$. It follows that
\[
 M_r\geq\frac{r^2}{4}J_1
       -Cr^3\sum_jK_j\HH_g^{n-1}(Z\cap B_j).
\]
Using $J_1=m/h^2$ from \eqref{eq:dong-total} and the sum in
\eqref{eq:transport-hessian-sum}, we get
\begin{equation}\label{eq:transport-lower}
 M_r\geq m\left(\frac{s^2}{4}-Cs^3\right).
\end{equation}
The unused parts of $\mu_+$ and $\mu_-$ have the same mass. If it is positive, we complete $\pi_0$ to a coupling by adding their product divided by that mass. If it is zero, $\pi_0$ is already a coupling. The completed coupling assigns mass at least $M_r$ to pairs at distance at most $r$. Dividing \eqref{eq:transport-upper} and \eqref{eq:transport-lower} by $m$ proves the theorem.
\end{proof}

Choosing $s$ so that the upper bound in \eqref{eq:transport-upper} is below $m/2$ gives another proof of the $W_1$ lower bound. Every coupling then moves at least half its mass farther than $sh$.
\subsection{Volume from weighted tube estimates}

\begin{proposition}\label{prop:concentration-volume}
For $0<p<q<\infty$ and $0<s<s_0$,
\begin{equation}\label{eq:concentration-volume}
 \min_\pm\Vol_g(T_{sh}^\pm)\geq cs
 \left(\frac{\norm{\varphi}_p}{\norm{\varphi}_q}\right)^{pq/(q-p)},
\end{equation}
where $c,s_0>0$ depend only on $(M,g),p,q$. In particular, if
$\kappa>0$ and $\norm{\varphi}_1\geq\kappa\norm{\varphi}_2$, then
$\min_\pm\Vol_g(T_\delta^\pm)\geq c\kappa^2\delta\sqrt\lambda$
for $0<\delta<s_0h$.
\end{proposition}

\begin{proof}
On either side $U=T_{sh}^\pm$, H\"older's inequality and Theorems~\ref{thm:moments} and \ref{thm:weighted} at $f=1$ give
\[
 cs^{p+1}\norm{\varphi}_p^p
 \leq \int_U|\varphi|^p\,dV_g
 \leq \Vol_g(U)^{1-p/q}\left(\int_U|\varphi|^q\,dV_g\right)^{p/q}
 \leq C\Vol_g(U)^{1-p/q}s^{(q+1)p/q}\norm{\varphi}_q^p .
\]
After division by $s^{(q+1)p/q}$, the power of $s$ on the left is
$1-p/q$. Raising both sides to the power $q/(q-p)$ gives
\eqref{eq:concentration-volume}.
\end{proof}
On the square flat torus of area one, Zygmund's estimate \cite[Section~3(d), p.~193]{Zyg} and Parseval's identity give $\|\varphi\|_4\le 5^{1/4}\|\varphi\|_2$. Interpolation yields
\[
 \|\varphi\|_2^3\le \|\varphi\|_1\|\varphi\|_4^2
 \le \sqrt5\,\|\varphi\|_1\|\varphi\|_2^2 .
\]
Thus \eqref{eq:concentration-volume} gives $\min_\pm\Vol_g(T_\delta^\pm)\ge c\,\delta/h$ for $0<\delta<s_0h$, uniformly in the eigenfunction. On this torus one can also deduce the estimate from the full-tube bound of Jakobson and Mangoubi \cite[Theorem~1.2, p.~1109]{JM}. The nodal set has area zero, and a translation exchanges the two sides of the tube \cite[proof of Theorem~2.5, p.~1375]{MT}.

\subsection{Transport and the volume of either side of a tube}

Let $f$ be continuous, real, nonzero and of mean zero. Put $m_f=\norm f_1/2$ and $T_\delta^\pm(f)=\{x:d_g(x,\{f=0\})<\delta,\ \pm f(x)>0\}$. For $0<\alpha\le1$ the least volume carrying a fraction $\alpha$ of the unsigned mass is
\[
 V_\alpha(f)=\inf\left\{\Vol_g(E):
       E\subset M\ \text{measurable},\
       \int_E |f|\,dV_g\geq\alpha\norm{f}_1\right\}.
\]

\begin{proposition}\label{prop:transport-volume}
Let $R>0$ and $0<\theta\leq1$. Suppose that a measure $\gamma$ on
$M\times M$ has first marginal at most $f^+dV_g$, second marginal
at most $f^-dV_g$, mass at least $\theta m_f$, and is supported where
$d_g(x,y)\leq R$. Then for every $\delta>0$,
\begin{equation}\label{eq:transport-volume}
 \min_{\pm}\Vol_g(T_\delta^\pm(f))
 \geq c_g\min\{1,\delta/R\}\,V_{\theta/2}(f).
\end{equation}
\end{proposition}

\begin{proof}
By monotonicity we may take $\delta\le R$. For $R$ below a fixed geometric radius, cover $M$ by $R$-balls $B_i$ with bounded overlap of $4B_i$. Let $E_i$ be whichever of $2B_i\cap\{f>0\}$ and $2B_i\cap\{f<0\}$ has smaller volume, choosing either in a tie. Put $E=\bigcup_i E_i$.

For $\gamma$-almost every $(x,y)$, the two points have opposite signs and $d_g(x,y)\leq R$. Choose $B_i$ containing $x$. Then $y$ belongs to $2B_i$, and at least one endpoint belongs to $E_i$. The marginal bounds give
\[
 \theta m_f\leq
 \int_{M\times M}\bigl(\mathbf 1_E(x)+\mathbf 1_E(y)\bigr)\,d\gamma
 \leq \int_E |f|\,dV_g .
\]
Thus $|E|\geq V_{\theta/2}(f)$, where $|\cdot|$ denotes volume.

To estimate the positive side, take $u=\min\{1,\dist_g(\,\cdot\,,\{f\le0\})/\delta\}$. At a positive point the distance to $\{f\le0\}$ is the distance to $\{f=0\}$. Thus $|\nabla u|\le\delta^{-1}\mathbf1_{T_\delta^+(f)}$ almost everywhere.

If $|T_\delta^+(f)\cap2B_i|\ge|E_i|/2$, we already have $|E_i|\le2(R/\delta)|T_\delta^+(f)\cap4B_i|$. Otherwise $u=1$ on a positive set of volume at least $|E_i|/2$ in $2B_i$, and $u=0$ on a negative set of volume at least $|E_i|$ there. Write $\bar u_i$ for its average over $2B_i$. The local $L^1$ Poincar\'e inequality gives
\[
 \tfrac12 |E_i|\leq \int_{2B_i}|u-\bar u_i|\,dV_g
 \leq C_g R\int_{4B_i}|\nabla u|\,dV_g
 \leq C_g(R/\delta)|T_\delta^+(f)\cap4B_i|.
\]
We sum over $i$ and use bounded overlap together with $V_{\theta/2}(f)\le |E|\le\sum_i|E_i|$. Replacing $f$ by $-f$ gives the estimate for the negative side. If $R$ is larger, we can carry out the argument on $M$ itself, taking $E$ to be the smaller sign set. The constant in the global $L^1$ Poincar\'e inequality is then at most $C_gR$.
\end{proof}

For an eigenfunction, Theorem~\ref{thm:cancellation} gives a coupling of
$\mu_+$ and $\mu_-$ supported on pairs at distance at most $C_g h$. Taking $\theta=1$ gives
\[
 \min_\pm\Vol_g(T_\delta^\pm)
 \geq c_g\min\{1,\delta/h\}\,V_{1/2}(\varphi).
\]
For $1<q\leq\infty$ and $1/q+1/q'=1$, H\"older's inequality gives
$V_\alpha(f)\geq(\alpha\norm{f}_1/\norm{f}_q)^{q'}$. For $p=1$
this recovers the power of the norm ratio in
\eqref{eq:concentration-volume}, and also allows $q=\infty$.

If the frequencies of $f$ are at least $\kappa>0$, the coupling from Theorem~\ref{thm:spectral-sums} assigns at least half its mass to pairs at distance at most $C_g/\kappa$. If the sum $S$ of distinct wavelengths is finite, it is supported on pairs at distance at most $C_g S$. Taking $\theta=1/2$ and $1$ gives the respective lower bounds
\[
 c_g\min\{1,\delta\kappa\}\,V_{1/4}(f)
 \qquad\hbox{and}\qquad c_g\min\{1,\delta/S\}\,V_{1/2}(f)
\]
for $\min_\pm\Vol_g(T_\delta^\pm(f))$.

There is a corresponding estimate for $Z_f=\{f=0\}$. We may assume $\HH_g^{n-1}(Z_f)<\infty$. Then $\{f>0\}$ has finite perimeter and its reduced boundary is contained in $Z_f$. In each ball $2B_i$ from the proof of Proposition~\ref{prop:transport-volume}, both $\{f>0\}$ and its complement have volume at least $|E_i|$. Relative isoperimetry gives
\[
 |E_i|\le C_gR\,|E_i|^{(n-1)/n}
       \le C_g R\,\HH_g^{n-1}(Z_f\cap 2B_i).
\]
Summing gives $\HH_g^{n-1}(Z_f)\ge c_g V_{\theta/2}(f)/R$. For larger $R$ the same argument uses the global isoperimetric inequality. In particular, for continuous $f$ with frequencies at least $\kappa>0$,
\[
 \HH_g^{n-1}(Z_f)\ge c_g\kappa\,V_{1/4}(f).
\]
Since $V_{1/4}(f)\ge\|f\|_1/(4\|f\|_\infty)$, this includes a bound of the form obtained from the transport--perimeter inequality of Cavalletti and Farinelli \cite[Theorem~1.1, (1.3), p.~3]{CF} and the $W_1$ estimate of \cite[Theorem~3, p.~1160]{CMO}.

\section{Cancellation for sums of eigenfunctions}\label{sec:spectral-sums}

\begin{proof}[Proof of Theorem~\ref{thm:spectral-sums}]
On the flat torus Steinerberger used ball averages at Bessel zeros to remove one shell of frequencies at a time \cite[pp.~260--261]{StSign}. Here we choose each radius so that the multiplier vanishes on the next eigenspace in the list and the $L^1$ norm of the averaged function is at most half the previous norm. To prove the contraction, compare the kernels at two centres. By \eqref{eq:ball-kernel}, for $0<r<r_g$ and $x,x'$ outside a null set depending on $r$,
\[
 \int_M|k_r(x,y)-k_r(x',y)|\,dV_g(y)
     \leq C_g\bigl(d_g(x,x')/r+r^2\bigr).
\]
Replacing each kernel by $(\omega_n r^n)^{-1}\mathbf1_{B_g(x,r)}$ gives integrated error $O_g(r^2)$. For $d=d_g(x,x')<r$ the symmetric difference lies in the annulus of radii $r\pm d$ about $x$, of volume at most $C_g r^{n-1}d$ in normal coordinates after decreasing $r_g$. For $d\ge r$ use the sum of the ball volumes.

Now let $u\in L^2(M)$ be real with frequencies at least $\kappa>0$.
By \cite[Theorem~3, p.~1160]{CMO} with eigenvalue cutoff $L=\kappa^2$, there is a coupling $\gamma$ of $u^+dV_g$
and $u^-dV_g$ with
$\int d_g(x,x')\,d\gamma(x,x')\leq C_g\kappa^{-1}\norm{u}_1$.
Both marginals give zero mass to the exceptional centres. By symmetry,
$K_ru(y)$ is the integral of $k_r(x,y)-k_r(x',y)$ against $\gamma$.
The preceding estimate therefore gives
\begin{equation}\label{eq:spectral-contraction}
 \norm{K_ru}_1\leq C_g\bigl((r\kappa)^{-1}+r^2\bigr)\norm{u}_1.
\end{equation}

Choose a positive zero $a$ of $F$ large enough that $C_g/a\leq1/4$
in \eqref{eq:spectral-contraction}. We may assume $ah_1<r_g$ and
$C_g(ah_1)^2\leq1/4$, since otherwise $\diam(M)\leq C_gh_1$ and
any coupling suffices. Starting with $u_0=f$, set $r_i=ah_i$ and
$u_i=K_{r_i}u_{i-1}$ for each positive $h_i$. Since $K_{r_i}$ commutes
with the Laplacian and $F(r_i\kappa_i)=0$, its multiplier vanishes on
the eigenspace of eigenvalue $\kappa_i^2$. All frequencies of $u_{i-1}$
are therefore at least $\kappa_i$. Applying \eqref{eq:spectral-contraction}
at each step gives
$\norm{u_i}_1\leq 2^{-i}\norm f_1$. If there are $k$ frequencies,
then $u_k=0$.

For each $\mu_\pm$, keep a measure on the original and current positions, starting with the image of $\mu_\pm$ under $x\mapsto(x,x)$. At step $i$ apply $K_{r_i}$ to the second coordinate, giving second marginals $\beta_\pm$. Let $\alpha_z^\pm$ be the conditional distributions of the original points given the current position $z$.

With $\eta=\beta_+\wedge\beta_-$ the largest measure dominated by both, set
\[
 \pi_i=\int_M \alpha_z^+\otimes\alpha_z^-\,d\eta(z).
\]
Retain the same conditional distributions with second marginals $\beta_\pm-\eta$ for the next step. Averaging preserves each mass; subtracting $\eta$ leaves the signed difference $u_i\,dV_g$. The remaining marginals are mutually singular, so they are $u_i^\pm\,dV_g$, each of mass at most $2^{-i}m_f$.

The remaining mass tends to zero and vanishes at step $k$ if there are only $k$ frequencies. Thus $\pi=\sum_i\pi_i$ has marginals $\mu_+$ and $\mu_-$. Each original point contributing to $\pi_i$ lies within $\sum_{j\leq i}r_j=a S_i$ of $z$, so $\pi_i$ is supported on pairs at distance at most $2a S_i$. This gives \eqref{eq:spectral-tail} and the support bound when $\sum_i h_i$ is finite.

For $1\leq p<\infty$ let $\mathbb P(I=i)=\pi_i(M\times M)/m_f$. Then $\mathbb P(I>i)\leq 2^{-i}$ gives $(\mathbb E I^p)^{1/p}\leq Cp$. For $k=\lceil p\rceil$, the decreasing wavelengths give $S_i\leq S_k\max\{1,i/k\}$. The $L^p$ norm of $d_g$ under $\pi/m_f$ is at most $2a(\mathbb E S_I^p)^{1/p}\leq 2aS_k(1+Cp/k)$, proving \eqref{eq:spectral-wasserstein}.
\end{proof}

For a finite sum, every ball of radius greater than $C_g\sum_i h_i$
contains a zero. Indeed, if $f$ had one strict sign on such a ball,
a smaller concentric ball would carry mass at distance greater than
$C_g\sum_i h_i$ from the opposite sign.

For arbitrary spectral support, $S_\ell\leq\ell/\kappa_1$. Hence
$\pi\{d_g(x,y)>t/\kappa_1\}\leq C_g e^{-c_gt}m_f$ and
$W_p(\nu_+,\nu_-)\leq C_gp/\kappa_1$. This removes the logarithmic
factor from the high-frequency estimate of \cite[Theorem~5, p.~318]{StFourier}
for $1<p<\infty$. The case $p=1$ was proved in \cite[Theorem~3, p.~1160]{CMO}.

To see why the bound must grow at least linearly with $p$, take a function of one circle coordinate on the standard flat torus,
\[
 f_N(x)=\operatorname{Re}\left[
 e^{iNx}\sum_{j=0}^{N-1}(-1)^j\binom{N+j-1}{j}(e^{ix}-1)^j\right].
\]
Its absolute frequencies lie between $N$ and $2N-1$. The sum is the
Taylor polynomial of $z^{-N}$ about $1$, evaluated at $z=e^{ix}$.
For $|x|\leq b$, where $b>0$ is a sufficiently small absolute
constant, the first omitted term is bounded by $(4b)^N$ and each
successive term has absolute value at most $2b$ times that of its
predecessor. Thus $f_N\geq1/2$ throughout the arc. On the whole
circle $\norm{f_N}_\infty\leq8^N$, since the binomial coefficients
are bounded by $4^N$ and $\sum_{j<N}2^j<2^N$. The smaller arc
$|x|\leq b/2$ therefore carries a fraction at least $c\,8^{-N}$ of
the positive mass. All this mass must travel at least $b/2$ to reach
the negative set. For the normalized sign measures this gives
$W_p\geq(b/2)(c\,8^{-N})^{1/p}$. Taking $N=\lceil p\rceil$ yields
$\kappa_1 W_p\geq c'p$.

This example shows that the bound $C_g\sum h_i$ cannot be improved in order. All frequencies from $N$ to $2N-1$ occur, since the nonzero contributions to the coefficient of $z^{N+k}$ have sign $(-1)^k$. Thus $\sum h_i=\sum_{j=N}^{2N-1}j^{-1}=\log 2+O(N^{-1})$, while the positive arc has length bounded below independently of $N$. Replacing $x$ by $Lx$ divides both lengths by $L$. For consecutive frequencies on the circle, Kozma and Oravecz found the exact supremum of the zero-free length \cite[Theorem~2, p.~448]{KO}.

If $f\ge0$ almost everywhere on $B(x,\rho)$, all positive mass in $B(x,\rho/2)$ must travel at least $\rho/2$. With $\kappa_1$ the least frequency, the exponential tail bound at $\rho/3$ gives
\[
 \int_{B(x,\rho/2)}f^+\,dV_g\leq C_g e^{-c_g\rho\kappa_1}m_f .
\]

\section*{Acknowledgements}
The author thanks Stefan Steinerberger, Emanuel Milman and Melchior Wirth for their conversations and correspondence. He is especially grateful to Alexander Logunov for the ideas that led to the proof of Theorem~\ref{old:nonconc} in dimensions $n\geq3$. The original research was partially supported by SEED Grant RD/0519-IRCCSH0-024. The author thanks the Indian Institute of Technology Bombay for providing ideal working conditions.

\end{document}